\documentclass[12pt]{amsart}
\usepackage{amsmath,fullpage}
\usepackage{amscd}
\usepackage{amssymb}
\usepackage{pdfsync}

\numberwithin{equation}{section}

\newtheorem*{coro}{Corollary} 
\newtheorem*{prop}{Proposition}
\newtheorem*{lem}{Lemma}
\newtheorem*{thm}{Theorem}
\newtheorem*{rmk}{Remark}

\theoremstyle{definition}

\newcommand{\cB}{{\mathcal B}}

\newcommand{\cO}{{\mathcal O}}

\newcommand{\cF}{{\mathcal F}}
\newcommand{\cN}{{\mathcal N}}
\newcommand{\cH}{{\mathcal H}}

\newcommand{\cP}{{\mathcal P}}
\newcommand{\cL}{{\mathcal L}}

\newcommand{\bP}{{\mathbb P}}
\newcommand{\bQ}{{\mathbb Q}}

\newcommand{\tr}{{\text{tr}}}

\newcommand{\Hom}{{\text{Hom}}}

\newcommand{\End}{{\text{End}}}
\newcommand{\Ad}{{\text{Ad}}}

\newcommand{\supp}{{\text{supp}}}

\newcommand{\rc}{{\mathrm c}}

\newcommand{\Lb}{{\mathfrak{b}}}
\newcommand{\Lt}{{\mathfrak{t}}}
\newcommand{\Lg}{{\mathfrak g}}
\newcommand{\Lo}{{\mathfrak o}}
\newcommand{\Ln}{{\mathfrak{n}}}

\newcommand{\Ll}{{\mathfrak{l}}}
\newcommand{\Lp}{{\mathfrak{p}}}

\newcommand{\bX}{{\mathbf{X}}}
\newcommand{\bY}{{\mathbf{Y}}}

\newcommand{\tF}{{\textbf{F}}}
\newcommand{\tk}{{\textbf{k}}}

\newcommand{\p}{\perp}

\newcommand{\la}{\langle}
\newcommand{\ra}{\rangle}

\newcommand{\beq}{\begin{equation*}}
\newcommand{\eeq}{\end{equation*}}

\begin{document}
\title[Combinatorics of Springer correspondence]
{Combinatorics of the Springer correspondence for classical Lie algebras and their duals in characteristic 2}
        \author{Ting Xue}
        \address{Department of Mathematics, Massachusetts Institute of Technology,
Cambridge, MA 02139, USA}
        \email{txue@math.mit.edu}
\maketitle
\begin{abstract}
We give a combinatorial description of the Springer correspondence
for classical Lie algebras of type $B,C$ or $D$ and their duals in
characteristic 2. The combinatorics used here is of the same kind as
those appearing in the description of (generalized) Springer
correspondence for unipotent case of classical groups by Lusztig in
odd characteristic and by Lusztig and Spaltentstein in
characteristic 2.
\end{abstract}
\section{introduction}
Let $G$ be a connected reductive algebraic group over an
algebraically closed field of characteristic $p$. Let $\Lg$ be the
Lie algebra of $G$ and $\Lg^*$ the dual vector space of $\Lg$. When
$p$ is large enough, Springer \cite{Sp1} constructs a correspondence
which associates to an irreducible character of the Weyl group of
$G$ a unique pair $(x,\phi)$ with $x\in\Lg$ nilpotent and $\phi$ an
irreducible character of the component group
$A_G(x)=Z_G(x)/Z_G^0(x)$. For arbitrary $p$, Lusztig \cite{Lu1}
constructs the generalized Springer correspondence which is related
to unipotent conjugacy classes in $G$. Assume $p=2$ and $G$ is of
type $B,C$ or $D$, a Springer correspondence for $\Lg$ (resp.
$\Lg^*$) is constructed in \cite{X1} (resp. \cite{X2}) using a
similar construction as in \cite{Lu1,Lu4}.

Assume $G$ is classical. When $p$ is large, Shoji \cite{sh1}
describes an algorithm to compute the Springer correspondence which
does not provide a closed formula. A combinatorial description of the
generalized correspondence for $G$ is given by Lusztig \cite{Lu1}
for $p\neq 2$  and by Lusztig, Spaltenstein \cite {LS} for $p=2$.
Spaltenstein \cite{Spal} describes a part of the Springer
correspondence for $\Lg$ (when $p=2$) under the assumption that the
theory of Springer representation is still valid in this case. We
describe the Springer correspondence for $\Lg$ and $\Lg^*$  using
similar combinatorics that appears in \cite{Lu1,LS}. It is very nice
that this combinatorics gives a unified description for
(generalized) Springer correspondences of classical groups in all
cases, namely, in $G$, $\Lg$ and $\Lg^*$ in all characteristics.
Moreover it gives rise to closed formulas for computing the
correspondences.

\section{Recollections and outline}

\subsection{}\label{ssec-n1} Throughout this paper, $\tk$
denotes an algebraically closed field of characteristic 2 unless
otherwise stated, $G$ denotes a connected algebraic group of type
$B,C$ or $D$ over $\tk$, $\Lg$ the Lie algebra of $G$ and $\Lg^*$
the dual vector space of $\Lg$. There is a natural coadjoint action
of $G$ on $\Lg^*$, $g.\xi(x)=\xi(\Ad(g^{-1})x)$ for $g\in
G,\xi\in\Lg^*$, $x\in\Lg$, where $\Ad$ is the adjoint action of $G$
on $\Lg$.

\subsection{} For a finite group $H$, we denote $H^\wedge$ the set
of irreducible characters of $H$.

Denote $\mathbf{W}_G$  the Weyl group of $G$ and $\mathfrak{A}_\Lg$
(resp. $\mathfrak{A}_{\Lg^*}$) the set of all pairs
$(\mathrm{c},\mathcal{F})$ with $\mathrm{c}$ a nilpotent $G$-orbit
in $\Lg$ (resp. $\Lg^*$) and $\mathcal{F}$ an irreducible
$G$-equivariant local system on $\mathrm{c}$ (up to isomorphism).
Then $\mathfrak{A}_\Lg$ (resp. $\mathfrak{A}_{\Lg^*}$) can be identified with
 the set of all pairs $(x,\phi)$ (resp. $(\xi,\phi)$) with
$x\in\Lg$ (resp. $\xi\in\Lg^*$) nilpotent (up to $G$-action) and
$\phi\in A_G(x)^\wedge$ (resp. $\phi\in A_G(\xi)^\wedge$), where
$A_G(x)=Z_G(x)/Z_G^0(x)$, $A_G(\xi)=Z_G(\xi)/Z_G^0(\xi)$,
$Z_G(x)=\{g\in G|\Ad(g)x=x\}$ and $Z_G(\xi)=\{g\in G|g.\xi=\xi\}$.

\subsection{}\label{ssec-isogeny}A Springer correspondence for $\Lg$
(resp. $\Lg^*$) is constructed in \cite{X1} (resp. \cite{X2})
assuming $G$ is adjoint (resp. simply connected). The correspondence
is a bijective map from $\mathfrak{A}_{\Lg}$ (resp.
$\mathfrak{A}_{\Lg^*}$) to $\mathbf{W}_G^\wedge$. This induces a
Springer correspondence for any $\Lg$ (resp. $\Lg^*$) as in
\ref{ssec-n1}. In fact, there are natural bijections between the
sets of nilpotent orbits in two Lie algebras (resp. duals of the Lie
algebras) of groups in the same isogeny class, and moreover,
corresponding component groups of centralizers are isomorphic. Hence
the sets $\mathfrak{A}_{\Lg}$ (resp. $\mathfrak{A}_{\Lg^*}$) are
naturally identified in each isogeny class.

\subsection{}\label{ss1-1} For a Borel
subgroup $B$ of $G$, we write $B=TU$ a Levi decomposition of $B$ and
denote $\Lb$, $\Lt$ and $\Ln$ the Lie algebra of $B,T$ and $U$
respectively. We define $\Ln^*=\{\xi\in\Lg^*|\xi(\Lb)=0\}$ and
$\Lb^*=\{\xi\in\Lg^*|\xi(\Ln)=0\}$.

For a parabolic subgroup $P$ of $G$, we denote $U_P$ the unipotent
radical of $P$,  $\Lp$ and $\Ln_P$ the Lie algebra of $P$ and $U_P$
respectively. For a Levi subgroup $L$ of $P$, we denote $\Ll$ the
Lie algebra of $L$. Define $\Lp^*=\{\xi\in\Lg^*|\xi(\Ln_P)=0\}$,
$\Ln_P^*=\{\xi\in\Lg^*|\xi(\Ll\oplus\Ln_P)=0\}$ and
$\Ll^*=\{\xi\in\Lg^*|\xi(\Ln_P\oplus\Ln_P^-)=0\}$ where
$\Lg=\Ll\oplus\Ln_P\oplus\Ln_P^-$. We have
$\Lp^*=\Ll^*\oplus\Ln_P^*$.

\subsection{}\label{sec-res}
Let $P$ be a parabolic subgroup of $G$ with a Levi decomposition
$P=LU_P$, where ${}_{s.s}\text{rank}(L)={}_{s.s}\text{rank}(G)-1$ (${}_{s.s}\text{rank}$ denotes the semisimple rank). We identify $L$
with $P/U_P$ and $\Ll$ with $\Lp/\Ln_P$. Let $x\in\Lg$ and
$x'\in\Ll$ be nilpotent elements. Consider the variety
$$Y_{x,x'}=\{g\in G|\Ad(g^{-1})(x)\in x'+\Ln_P\}.$$
The group $Z_G(x)\times Z_L(x')U_P$ acts on $Y_{x,x'}$ by
$(g_0,g_1).g=g_0gg_1^{-1}$.

Let $d_{x,x'}=(\dim Z_G(x)+\dim Z_L(x'))/2+\dim\Ln_P$. We have $\dim
Y_{x,x'}\leq d_{x,x'}$ (see Proposition \ref{prop-dim} (ii)). Let
$S_{x,x'}$ be the set of all irreducible components of $Y_{x,x'}$ of
dimension $d_{x,x'}$. Then the group $A_G(x)\times A_L(x')$ acts on
$S_{x,x'}$. Denote $\varepsilon_{x,x'}$ the corresponding
representation.  We prove in section \ref{sec-rf} the following
restriction formula
\begin{equation*}\label{resfor}
(\mathbf{R})\qquad\la \phi\otimes\phi',\varepsilon_{x,x'} \ra=\la
\text{Res}^{\mathbf{W}_G}_{\mathbf{W}_{L}}\rho^G_{x,\phi},\rho^L_{x',\phi'}\ra_{\mathbf{W}_{L}},
\end{equation*}
where $\phi\in A_G(x)^\wedge,\phi'\in A_L(x')^\wedge$, and
$\rho^G_{x,\phi}\in \mathbf{W}_G^\wedge$, $\rho^L_{x',\phi'}\in
\mathbf{W}_L^\wedge$ correspond to the pairs
$(x,\phi)\in\mathfrak{A}_\Lg,(x',\phi')\in\mathfrak{A}_\Ll$
respectively under the Springer correspondence.

It suffices to consider the case where $G$ is adjoint (see
\ref{ssec-isogeny}). The proof is essentially the same as that of
the restriction formula in unipotent case \cite{Lu1}.

\subsection{} We preserve the notations from \ref{sec-res}. Let $\xi\in\Lg^*$ and
$\xi'\in\Ll^*$ be nilpotent elements. We define
$Y_{\xi,\xi'},S_{\xi,\xi'},\varepsilon_{\xi,\xi'}$ as
$Y_{x,x'},S_{x,x'},\varepsilon_{x,x'}$ replacing $x,x'$,$\Lp,\Ln_P$
by $\xi,\xi'$,$\Lp^*,\Ln_P^*$ respectively and adjoint $G$-action on
$\Lg$ by coadjoint $G$-action on $\Lg^*$. We identify $\Ll^*$ with
$\Lp^*/\Ln_P^*$. We have the following restriction formula
\begin{equation*}\label{resfor}
(\mathbf{R}')\qquad\la \phi\otimes\phi',\varepsilon_{\xi,\xi'}
\ra=\la
\text{Res}^{\mathbf{W}_G}_{\mathbf{W}_{L}}\rho^G_{\xi,\phi},\rho^L_{\xi',\phi'}\ra_{\mathbf{W}_{L}},
\end{equation*}
where $\phi\in A_G(\xi)^\wedge,\phi'\in A_L(\xi')^\wedge$, and
$\rho^G_{\xi,\phi}\in \mathbf{W}_G^\wedge,\rho^L_{\xi',\phi'}\in
\mathbf{W}_L^\wedge$ correspond to the pairs
$(\xi,\phi)\in\mathfrak{A}_{\Lg^*},(\xi',\phi')\in\mathfrak{A}_{\Ll^*}$
respectively under the Springer correspondence. The proof of
$(\mathbf{R}')$ is entirely similar to that of $(\mathbf{R})$ and is
omitted.

\subsection{}\label{ssec-spd}The reference for this subsection and \ref{ssec-or} is \cite{X2}.
Let $V$ be a vector space of dimension $2n$ over $\tk$ equipped with
a non-degenerate symplectic form $\mathbf{\beta}:V\times
V\rightarrow \tk$. The symplectic group is defined as
$Sp(2n)=Sp(V)=\{g\in GL(V)\ |\ \beta(gv,gw)=\beta(v,w), \forall\
v,w\in V\}$ and its Lie algebra is
$\mathfrak{sp}(2n)=\mathfrak{sp}(V)=\{x\in \mathfrak{gl}(V)\ |\
\beta(xv,w)+\beta(v,xw)=0, \forall\ v,w\in V\}$.

 Recall that for a nilpotent
element $\xi\in \mathfrak{sp}(2n)^*$, we associate a well-defined
quadratic form $\alpha_\xi:V\to\tk$, $\alpha_\xi(v)=\beta(v,Xv)$ and
a nilpotent endomorphism $T_\xi:V\to V$,
$\beta(T_\xi(v),w)=\beta_\xi(v,w)$, where $X\in\End(V)$ is such that
$\xi(-)=\tr(X-)$ and $\beta_\xi$ is the bilinear form associated to
$\alpha_\xi$, namely,
$\beta_\xi(v,w)=\alpha_\xi(v+w)-\alpha_\xi(v)-\alpha_\xi(w)$ for $v,w\in V$.
Moreover, we define the function $\chi_\xi:\mathbb{N}\to\mathbb{N}$
by $\chi_\xi(m)=\min\{k\geq 0|\forall\ v\in
V,T_\xi^mv=0\Rightarrow\alpha_\xi(T_\xi^kv)=0\}$.

\subsection{}\label{ssec-or}
Let $V$ be a vector space of dimension $N$ over $\tk$ equipped with
a non-degenerate quadratic form $\alpha:V\rightarrow \tk$. Let
$\beta:V\times V\rightarrow\tk$ be the bilinear form associated to
$\alpha$, namely, $\beta(v,w)=\alpha(v+w)-\alpha(v)-\alpha(w)$ for all $v,w\in V$. The
orthogonal group is defined as $O(N)=O(V)=\{g\in GL(V)\ |\
\alpha(gv)=\alpha(v), \forall\ v \in V\}$ and its Lie algebra is
$\Lo(N)=\Lo(V)=\{x\in \mathfrak{gl}(V)\ |\ \beta(xv,v)=0, \forall\
v\in V\text{ and tr}(x)=0\}$. We define $SO(N)$ to
be the identity component of $O(N)$. 

Recall that for a nilpotent element $\xi\in \mathfrak{o}(2n+1)^*$,
we associate a well-defined bilinear form $\beta_\xi:V\times
V\to\tk$, $\beta_\xi(v,w)=\beta(Xv,w)+\beta(v,Xw)$ where
$X\in\End(V)$ is such that $\xi(-)=\tr(X-)$. Let $v_i$,
$i=0,\ldots,m$, be the uniquely determined set of
vectors by $\xi$ (see \cite[3.2]{X2}, we assume that $\alpha(v_m)=1$). If $m\geq 1$ let $u_i$,
$i=0,\ldots,m-1$ be a set of vectors as in \cite[Lemma 3.6]{X2}. Let
$V_{2m+1}$ be the vector space spanned by
$u_i,i=0,\ldots,m-1,v_i,i=0,\ldots,m$, $W$ a complementary vector
space of $V_{2m+1}$ in $V$ if $m=0$ and $W=\{v\in
V|\beta(v,V_{2m+1})=\beta_\xi(v,V_{2m+1})=0\}$ if $m\geq 1$. Then
$V=V_{2m+1}\oplus W$. Recall that $\beta$ is nondegenerate on $W$.
Define $T_\xi:W\to W$ by $\beta(T_\xi(w),w')=\beta_\xi(w,w')$ and a
function $\chi_W:\mathbb{N}\to\mathbb{N}$ by $\chi_W(s)=\min\{k\geq
0|\forall\ v\in V,T_\xi^sv=0\Rightarrow\alpha(T_\xi^kv)=0\}$.

Assume $m\geq 1$. Note that the set $\{u_i\}_{i=0}^{m-1}$ and $W$
depend on the choice of $u_0$ and are uniquely determined by $u_0$.
Suppose we take $\tilde{u}_0=u_0+w_0$, where $w_0\in W$ and
$\alpha(w_0)=0$, then $\tilde{u}_i=u_i+T_\xi^iw_0$ defines another
set of vectors as in \cite[Lemma 3.6]{X2}. Let $\tilde{V}_{2m+1}$,
$\tilde{W}$, $\tilde{T}_\xi$ be defined as $V_{2m+1}$, $W$, $T_\xi$
replacing $u_i$ by $\tilde{u}_i$. Then $V=\tilde{V}_{2m+1}\oplus
\tilde{W}$ and
$\tilde{W}=\{\sum_{i=0}^m\beta(w,T_\xi^iw_0)v_i+w|w\in W\}.$ On
$\tilde{W}$, we have
$\tilde{T}_\xi(\sum_{i=0}^m\beta(w,T_\xi^iw_0)v_i+w)=\sum_{i=0}^m\beta(w,T_\xi^{i+1}w_0)v_i+T_\xi
w$.

From now on, we always assume the decomposition $V=V_{2m+1}\oplus W$
is a normal form of $\xi$, namely, if
$W=W_{\chi(\lambda_1)}(\lambda_1)\oplus\cdots\oplus
W_{\chi(\lambda_s)}(\lambda_s)$ with
$\lambda_1\geq\cdots\geq\lambda_s$ (notation as in \cite{X2}), then
$m\geq \lambda_1-\chi(\lambda_1)$.

\subsection{}For $n\geq 1$, let $\mathbf{W}_n$ denote a Weyl group of type
$B_n$ (or $C_n$). The set $\mathbf{W}_n^\wedge$ is parametrized by
ordered pairs of partitions $(\mu,\nu)$, $\mu=(\mu_1\geq\mu_2\geq\cdots\geq 0)$,
$\nu=(\nu_1\geq\nu_2\geq\cdots\geq 0)$, with
$\sum\mu_i+\sum\nu_i=n$. We use the convention that the trivial
representation corresponds to $(\mu,\nu)$ with $\mu=(n)$ and the
sign representation corresponds to $(\mu,\nu)$ with $\nu=(1^n)$.

For $n\geq 2$, let $\mathbf{W}_n'\subset \mathbf{W}_n$ denote a Weyl
group of type $D_n$. Let $\mathbf{W}_0'=\mathbf{W}_1'=\{1\}$. Let
${\mathbf{W}_n^\wedge}'$ be the quotient of $(\mathbf{W}_n')^\wedge$
by the natural action of $\mathbf{W}_n/\mathbf{W}_n'$. The
parametrization of $\mathbf{W}_n$ by ordered pairs of partitions
induces a parametrization of ${\mathbf{W}_n^\wedge}'$ by unordered
pairs of partitions $\{\mu,\nu\}$. Moreover, $\{\mu,\nu\}$
corresponds to one (resp. two) element(s) of
$(\mathbf{W}_n')^\wedge$ if and only if $\mu\neq\nu$ (resp.
$\mu=\nu$). We say that $\{\mu,\nu\}$ and the corresponding elements
of ${\mathbf{W}_n^\wedge}'$ and $(\mathbf{W}_n')^\wedge$ are
non-degenerate (resp. degenerate).

\subsection{}\label{ssec-weylgp}

Assume  $G=Sp(2n)$ or $G=O(2n+1)$. The Springer correspondence for
$\Lg$ (resp. $\Lg^*$) is a bijective map
$$\gamma:\mathfrak{A}_\Lg\xrightarrow{\sim} \mathbf{W}_n^\wedge\text{ (resp. }
\gamma':\mathfrak{A}_{\Lg^*}\xrightarrow{\sim}
\mathbf{W}_n^\wedge).$$

Assume $G=SO(2n)$. The Springer correspondence for $\Lg$ (or
$\Lg^*$) is a bijective map
\begin{equation}\label{sp-1-o}\gamma:\mathfrak{A}_\Lg\cong\mathfrak{A}_{\Lg^*}\xrightarrow{\sim} (\mathbf{W}_n')^\wedge.\end{equation}
 Let
$\tilde{G}=O(2n)$. The group $\tilde{G}/G$ acts on
$\mathfrak{A}_{\Lg}$ and on the set of all nilpotent $G$-orbits in
$\Lg$. An element in $\mathfrak{A}_\Lg$ or a nilpotent orbit in
$\Lg$ is called non-degenerate (resp. degenerate) if it is fixed
(resp. not fixed) by this action. Then $(x,\phi)\in\mathfrak{A}_\Lg$
is degenerate if and only if $x$ is degenerate, in this case
$A_{G}(x)=1$ and thus $\phi=1$. Let $\tilde{\mathfrak{A}}_\Lg$ be
the quotient of $\mathfrak{A}_\Lg$ by $\tilde{G}/G$. Then
(\ref{sp-1-o}) induces a bijection
\begin{equation}\label{sp-2-o}\tilde{\gamma}:\tilde{\mathfrak{A}}_\Lg\xrightarrow{\sim} {\mathbf{W}_n^\wedge}'.\end{equation}

\subsection{}\label{ss1-2}  Assume $G=SO(V)$. Let
$\tilde{G}=O({V})$. Note that $\tilde{G}\neq G$ if and only if
$\dim(V)$ is even. Let $\Sigma\subset {V}$ be a line such that
$\alpha|_{\Sigma}=0$. Let $\tilde{P}$ be the stabilizer of $\Sigma$
in $\tilde{G}$ and $P$ the identity component of $\tilde{P}$. Then
$P$ is a parabolic subgroup of $G$. Let $L$ be a Levi subgroup of
$P$ and $\tilde{L}=N_{\tilde{P}}(L)$. Let $U_P$, $\Lp$, $\Ln_P$ be
as in \ref{ss1-1}. Then $\tilde{P}=\tilde{L}U_P$ and
$L=\tilde{L}^0$. Fix a Borel subgroup $B\subset P$ and let
$\tilde{B}=N_{\tilde{G}}(B)$. Denote
$\tilde{\cB}=\{g\tilde{B}g^{-1}|g\in \tilde{G}\}$,
$\tilde{\cP}=\{g\tilde{P}g^{-1}|g\in \tilde{G}\}$.

Let $x\in\Lg$ be nilpotent. Define
$\tilde{\cB}_x=\{g\tilde{B}g^{-1}\in\tilde{\cB}|\Ad(g^{-1})(x)\in\Lb\}$
and
$\tilde{\cP}_x=\{g\tilde{P}g^{-1}\in\tilde{\cP}_x|\Ad(g^{-1})(x)\in\Lp\}$.
The natural morphism
$\varrho_x:\tilde{\cB}_x\rightarrow\tilde{\cP}_x$,
$g\tilde{B}g^{-1}\rightarrow g\tilde{P}g^{-1}$ is $Z_{\tilde{G}}(x)$
equivariant. We have a well defined map
$$f_x:\tilde{\cP}_x\rightarrow\mathcal{CN}(\Lp/\Ln_P),
\ g\tilde{P}g^{-1}\mapsto\text{ orbit of }Ad(g^{-1})x+\Ln_P,$$ where
$\mathcal{CN}(\Lp/\Ln_P)$ is the set of nilpotent
$\tilde{P}/U_P$-orbits in $\Lp/\Ln_P$. Let $\rc'\in
f_x(\tilde{\cP}_x)$ be a nilpotent orbit. Define
$\mathbf{Y}=f_x^{-1}(\rc')$ and
$\mathbf{X}=\varrho_x^{-1}(\mathbf{Y})$.

We can assume $\tilde{P}\in \mathbf{Y}$. We identify $\tilde{L}$
with $\tilde{P}/U_P$, $\Ll$ with $\Lp/\Ln_P$. Let $x'$ be the image
of $x$ in $\Ll$ and
$\tilde{A}'(x')=A_{\tilde{L}}(x')=Z_{\tilde{L}}(x')/Z_{\tilde{L}}^0(x'),
H=Z_{\tilde{G}}(x)\cap
\tilde{P}=Z_{\tilde{P}}(x),K=Z_{\tilde{G}}^0(x)\cap \tilde{P}.$ The
natural morphisms $H\rightarrow Z_{\tilde{G}}(x)$, $H\rightarrow
Z_{\tilde{L}}(x')$ and $K\rightarrow Z_{\tilde{L}}(x')$ induce
morphisms $H\rightarrow A_{\tilde{G}}(x)$, $H\rightarrow
A_{\tilde{L}}(x')$ and $K\rightarrow A_{\tilde{L}}(x')$. Let
$\tilde{A}_P$ be the image of $H$ in $A_{\tilde{G}}(x)$ and
$\tilde{A}_P'$ be the image of $K$ in $A_{\tilde{L}}(x')$. Then we
have a natural morphism $\tilde{A}_P\to
\tilde{A}'(x')/\tilde{A}'_P$.

If $G=\tilde{G}$, then we omit the tildes from the notations, for
example, $A_P=\tilde{A}_P$ and etc.

\subsection{}\label{ssec-sxx}
We preserve the notations in \ref{ss1-2}. Let $\tilde{Y}_{x,x'}$ and
$\tilde{S}_{x,x'}$ be defined as in \ref{sec-res} replacing $G$ by
$\tilde{G}$ and $L$ by $\tilde{L}$. Note that
$\tilde{S}_{x,x'}\neq\emptyset$ if and only if $\dim
\bX=\dim\tilde{\cB}_x$, where $\bX$ is defined as in \ref{ss1-2}
with $\rc'$ the orbit of $x'$. If $\tilde{S}_{x,x'}\neq\emptyset$,
then $\tilde{Y}_{x,x'}$ is a single orbit under the action of
$Z_{\tilde{G}}(x)\times Z_{\tilde{L}}(x')U_P$ (see Proposition
\ref{prop-gps}). It follows that $\tilde{S}_{x,x'}$ is a single
$A_{\tilde{G}}(x)\times A_{\tilde{L}}(x')$-orbit. Hence
$\tilde{S}_{x,x'}=A_{\tilde{G}}(x)\times
A_{\tilde{L}}(x')/\tilde{H}_{x,x'}$ for some subgroup
$\tilde{H}_{x,x'}\subset A_{\tilde{G}}(x)\times A_{\tilde{L}}(x')$.
The subgroup $\tilde{H}_{x,x'}$ is described as follows.

If $A,B$ are groups, a subgroup $C$ of $A\times B$ is characterized
by the triple $(A_0,B_0,h)$ where $A_0=\text{pr}_1(C)$, $B_0=B\cap
C$ and $h:A_0\rightarrow N_B(B_0)/B_0$ is defined by $a\mapsto bB_0$
if $(a,b)\in C$. Then $\tilde{H}_{x,x'}$ is characterized by the
triple $(\tilde{A}_P,\tilde{A}_P',h)$, where $h$ is the natural
morphism $\tilde{A}_P\to A_{\tilde{L}}(x')/\tilde{A}'_P$ described
in \ref{ss1-2}.

Assume $G=SO(2n)$. The subset $S_{x,x'}$ of $\tilde{S}_{x,x'}$ is
the image in $\tilde{S}_{x,x'}$ of the subgroup of
$A_{\tilde{G}}(x)\times A_{\tilde{L}}(x')$ consisting of the
elements that can be written as a product of even number of
generators. This is also the image of $A_{{G}}(x)\times
A_{{L}}(x')$.

\subsection{}\label{ssec-gps-d}Assume $G=Sp(V)$, or $SO(V)$ with $\dim V$ odd.
The definitions in \ref{ss1-2} apply to $\Lg^*$ (if $G=Sp(V)$, there
are no conditions on the line $\Sigma$). Let
$\varrho_\xi,f_\xi,{A}_P,{A}_P'$ etc. be defined in this way. Then
${Y}_{\xi,\xi'}$, ${S}_{\xi,\xi'}$ are described in the same way as
${Y}_{x,x'}$, ${S}_{x,x'}$ in \ref{ssec-sxx}.

\subsection{}\label{ss1-3}
The correspondence for symplectic Lie algebras is determined by
Spaltenstein \cite{Spal} since in this case the centralizer of a
nilpotent element is connected and $\mathfrak{A}_\Lg=\{(\rc,1)\}$.
We rewrite his results in section \ref{sec-sy} using different
combinatorics and describe the Springer correspondence for
orthogonal Lie algebras in section \ref{sec-scor}.  The proof will
essentially be as in \cite{Lu1}, which is based on the restriction
formula $(\mathbf{R})$ and the following observation of Shoji: if
$n\geq 3$, an irreducible character of $\mathbf{W}_n$ (resp. a
nondegenerate irreducible character of $\mathbf{W}_n'$) is
completely determined by its restriction to $\mathbf{W}_{n-1}$
(resp. $\mathbf{W}_{n-1}'$). We need to study the representations
$\varepsilon_{x,x'}$, which require a description of the groups
$\tilde{A}_P$ and $\tilde{A}_P'$. Extending some methods in
\cite{Spal2}, we describe these groups for orthogonal Lie algebras,
duals of symplectic Lie algebras and duals of odd orthogonal Lie
algebras in section \ref{sec-or}, \ref{sec-dsp} and \ref{sec-dor}
respectively.

\subsection{} The Springer correspondence for
the duals of symplectic Lie algebras and orthogonal Lie algebras is
described in section \ref{sec-d-c}. The proofs are very similar to
the Lie algebra case and we omit much detail.

\section{Restriction formula}\label{sec-rf}

Assume $G$ is adjoint. Fix a Borel subgroup $B$ of $G$ and a maximal
torus $T\subset B$. Let $\cB$ be the variety of Borel subgroups of
$G$. A proof of the restriction formula in unipotent case is given
in \cite{Lu1}. The proof for nilpotent case is essentially the same.
For completeness, we include the proof here.

\subsection{}\label{prop-dim}We prove first a dimension formula following \cite{Lu1}. Let $\mathcal{P}$ be a $G$-conjugacy class of
parabolic subgroups of $G$.  For $P\in\cP$, let $\bar{P}=P/U_P$,
$\bar{\Lp}=\Lp/\Ln_P$ and $\pi_{\Lp}:\Lp\rightarrow\bar{\Lp}$ be the
natural projection. Let $\rc$ be a nilpotent $G$-orbit in $\Lg$.
Assume for each $P\in\cP$,  given a nilpotent $\bar{P}$-orbit
$\rc_{\bar{\Lp}}\subset\bar{\Lp}$ with the following property: for
any $P_1,P_2\in\cP$ and any $g\in G$ such that $P_2=gP_1g^{-1}$, we
have
$\pi_{\Lp_2}^{-1}(\rc_{\bar{\Lp}_2})=\Ad(g)(\pi_{\Lp_1}^{-1}(\rc_{\bar{\Lp}_1}))$.
Let
\begin{eqnarray*}
&&Z'=\{(x,P_1,P_2)\in
\mathfrak{g}\times\cP\times\cP|x\in\pi_{\Lp_1}^{-1}(\rc_{\bar{\Lp}_1})
\cap\pi_{\Lp_2}^{-1}(\rc_{\bar{\Lp}_2})\}.
\end{eqnarray*}
We have a partition
$Z'=\cup_{\mathcal{\mathcal{O}}}Z'_{\mathcal{\mathcal{O}}}$, where
${\mathcal{\mathcal{O}}}$ runs through the $G$-orbits  on
$\cP\times\cP$ and $Z'_\mathcal{O}=\{(x,P_1,P_2)\in
Z'|(P_1,P_2)\in\mathcal{O}\}$.

We denote $\nu_G$ the number of positive roots in $G$ and set
$\bar{\nu}=\nu_{\bar{P}}$ $(P\in\cP)$. Let $c=\dim \rc$ and
$\bar{c}=\dim \rc_{\bar{\Lp}}$ for $P\in\cP$.

\begin{prop}
$\mathrm{(i)}$ Given $P\in\cP$ and $\bar{x}\in \rc_{\bar{\Lp}}$, we
have
$\dim(\rc\cap\pi_{\Lp}^{-1}(\bar{x}))\leq\frac{1}{2}(c-\bar{c})$.

$\mathrm{(ii)}$ Given $x\in \rc$, we have
$\dim\{P\in\cP|x\in\pi_{\Lp}^{-1}(\rc_{\bar{\Lp}})\}\leq(\nu_G-\frac{c}{2})-(\bar{\nu}-\frac{\bar{c}}{2})$.

$\mathrm{(iii)}$ If $d_0=2\nu_G-2\bar{\nu}+\bar{c}$, then $\dim
Z'_\mathcal{O}\leq d_0$ for all $\mathcal{O}$. Hence $\dim Z'\leq
d_0$.
\end{prop}
\begin{proof}
We prove the proposition by induction on the dimension of the group.
Assume $\cP=\{G\}$, the proposition is clear. Thus we can assume
that $\cP$ is a class of proper parabolic subgroups of $G$ and that
the proposition holds when $G$ is replaced by a group of strictly
smaller dimension.

Consider the map $Z'_{\cO}\rightarrow\cO$, $(x,P_1,P_2)\mapsto
(P_1,P_2)$. We see that proving (iii) for $Z'_{\cO}$ is the same as
proving that for a fixed $(P',P'')\in\cO$, we have
\begin{eqnarray}
&&\dim \pi_{\Lp'}^{-1}(\rc_{\bar{\Lp}'})
\cap\pi_{\Lp''}^{-1}(\rc_{\bar{\Lp}''})\leq
2\nu_G-2\bar{\nu}+\bar{c}-\dim\cO.\label{d-2}
\end{eqnarray}
Choose Levi subgroups $L'$ of $P'$ and $L''$ of $P''$ such that $L'$
and $L''$ contain a common maximal torus. An element in $\Lp'\cap\Lp''$ can be written both in the form
$x'+n'$ ($x'\in\Ll',n'\in\Ln_{P'}$) and in the form $x''+n''$
($x''\in\Ll'',n''\in\Ln_{P''}$). It is easy to see that there are
unique elements
$z\in\Ll'\cap\Ll'',u''\in\Ll'\cap\Ln_{P''},u'\in\Ll''\cap\Ln_{P'}$,
such that $x'=z+u'',x''=z+u'$. Hence (\ref{d-2}) is equivalent to
{\small\begin{eqnarray}\label{d-6}
&&\dim\{(n',n'',u'',u',z)\in\Ln_{P'}\times\Ln_{P''}\times(\Ll'\cap\Ln_{P''})\times(\Ll''\cap\Ln_{P'})\times(\Ll'\cap\Ll'')|
\  u''+n'=u'+n'',\\
&&z+u''\in\rc_{\bar{\Lp}'},z+u'\in\rc_{\bar{\Lp}''}\}\leq
2\nu_G-2\bar{\nu}+\bar{c}-\dim\cO\nonumber.
\end{eqnarray}}
(We identify $\Ll'=\bar{\Lp}',\Ll''=\bar{\Lp}''$, and then view
$\rc_{\bar{\Lp}'}\subset\Ll',\rc_{\bar{\Lp}''}\subset\Ll''$.) When
$(u'',u')\in(\Ll'\cap\Ln_{P''})\times(\Ll''\cap\Ln_{P'})$ is fixed,
the variety $\{(n',n'')\in\Ln_{P'}\times\Ln_{P''}|u''+n'=u'+n''\}$
is isomorphic to $\Ln_{P'}\cap\Ln_{P''}$. Since
$\dim(\Ln_{P'}\cap\Ln_{P''})=2\nu_G-2\bar{\nu}-\dim\cO$, we see that
(\ref{d-6}) is equivalent to
\begin{eqnarray}\label{d-5}
&&\dim\{(u'',u',z)\in(\Ll'\cap\Ln_{P''})\times(\Ll''\cap\Ln_{P'})\times(\Ll'\cap\Ll'')|\
z+u''\in \rc_{\bar{\Lp}'},z+u'\in \rc_{\bar{\Lp}''}\}\leq \bar{c}.
\end{eqnarray}

By the finiteness of the number of nilpotent orbits, the projection
$\text{pr}_3$ of the variety in (\ref{d-5}) on the $z$-coordinate is
a union of finitely many orbits
$\hat{\rc}_1\cup\hat{\rc}_2\cup\cdots\cup\hat{\rc}_m$ in
$\Ll'\cap\Ll''$ (note that $z$ is nilpotent). The inverse image
under $\text{pr}_3$ of a point $z\in\hat{\rc}_i$ is a product of two
varieties of the type considered in (i) for a smaller group ($G$
replaced by $L'$ or $L''$), thus by the induction hypothesis it has
dimension
$\leq\frac{1}{2}(\bar{c}-\dim\hat{\rc}_i)+\frac{1}{2}(\bar{c}-\dim\hat{\rc}_i)$.
Hence $\dim\text{pr}_3^{-1}(\hat{\rc}_i)\leq\bar{c}$, $\forall\
1\leq i\leq m$. Then (\ref{d-5}) holds. This proves (iii).

We show that (ii) is a consequence of (iii). Let $Z'(\rc)$ be the
subset of $Z'$ defined by $Z'(\rc)=\{(x,P_1,P_2)\in Z'|x\in \rc\}$.
If $Z'(\rc)$ is empty then the variety in (ii) is empty and (ii)
follows. Hence we may assume that $Z'(\rc)$ is non-empty. From
(iii), we have $\dim Z'(\rc)\leq d_0$. Consider the map
$Z'(\rc)\rightarrow \rc,\ (x,P_1,P_2)\mapsto x$. Each fiber of this
map is a product of two copies of the variety in (ii). It follows
that the variety in (ii) has dimension equal to $\frac{1}{2}(\dim
Z'(\rc)-\dim
\rc)\leq\frac{1}{2}(d_0-c)=\nu_G-\bar{\nu}+\frac{\bar{c}}{2}-\frac{c}{2}$.
Then (ii) follows.

We show that (i) is a consequence of (ii). Consider the variety
$\{(x,P)\in \rc\times\cP|x\in\pi_{\Lp}^{-1}(\rc_{\bar{\Lp}})\}$. By
projecting it to the $x$-coordinate and using (ii), we see that it
has dimension $\leq\nu_G-\bar{\nu}+\frac{\bar{c}}{2}+\frac{c}{2}$.
If we project it to the $P$-coordinate, each fiber is isomorphic to
the variety $\rc\cap\pi_{\Lp}^{-1}(\rc_{\bar{\Lp}})$. Hence
$\dim(\rc\cap\pi_{\Lp}^{-1}(\rc_{\bar{\Lp}}))\leq\nu_G-\bar{\nu}
+\frac{\bar{c}}{2}+\frac{c}{2}-\dim\cP=\frac{c+\bar{c}}{2}$. Now
$\pi_{\Lp}$ maps $\rc\cap\pi_{\Lp}^{-1}(\rc_{\bar{\Lp}})$ onto
$\rc_{\bar{\Lp}}$ and each fiber is the variety in (i). Hence the
variety in (i) has dimension
$\leq\frac{c+\bar{c}}{2}-\bar{c}=\frac{c-\bar{c}}{2}$. The
proposition is proved.
\end{proof}

\subsection{}
Let $P\supset B$ be a parabolic subgroup of $G$ with Levi subgroup
$L$ such that $T\subset L$. Let $\mathbf{W}_L=N_L(T)/T$. Then
$\bar{\bQ}_l[\mathbf{W}_L]$ is in a natural way a subalgebra of
$\bar{\bQ}_l[\mathbf{W}_G]$.

Recall that we have the map (see \cite{X1})
\begin{equation*}
\pi:\widetilde{Y}=\{(x,gT)\in Y\times
G/T|\Ad(g^{-1})(x)\in\Lt_0\}\rightarrow Y, (x,gT)\mapsto x,
\end{equation*}
where $Y$, $\Lt_0$ is the set of regular semisimple elements in
$\Lg$, $\Lt$ respectively. Let \begin{equation*}Y_L=\bigcup_{g\in
L}\Ad(g)\Lt_{0},\ \widetilde{Y}_1=\{(x,gL)\in\Lg\times
G/L|\Ad(g^{-1})(x)\in Y_L\}. \end{equation*} Then $\pi$ factors as
$\widetilde{Y}\xrightarrow{\pi'}\widetilde{Y}_1\xrightarrow{\pi''}Y$,
where $\pi'$ is $(x,gT)\mapsto(x,gL)$ and $\pi''$ is $(x,gL)\mapsto
x$.  The map $\pi':\widetilde{Y}\rightarrow \widetilde{Y}_1$ is  a
principal bundle with group $\mathbf{W}_L$. It follows that $\End
(\pi'_!\bar{\bQ}_{l\widetilde{Y}})=\bar{\bQ}_l[\mathbf{W}_L]$ and
that we have a canonical decomposition
\begin{equation}
\pi'_!\bar{\bQ}_{l\widetilde{Y}}=\bigoplus_{\rho'\in
\mathbf{W}_L^{\wedge}}(\rho'\otimes(\pi'_!\bar{\bQ}_{l\widetilde{Y}})_{\rho'}),
\end{equation}
where
$(\pi'_!\bar{\bQ}_{l\widetilde{Y}})_{\rho'}=\Hom_{\bar{\bQ}_l[\mathbf{W}_L]}(\rho',\pi'_!\bar{\bQ}_{l\widetilde{Y}})$
is an irreducible local system on $\widetilde{Y}_1$. Recall that
\begin{equation*}
\pi_!\bar{\bQ}_{l\widetilde{Y}}=\bigoplus_{\rho\in
\mathbf{W}_G^{\wedge}}(\rho\otimes(\pi_!\bar{\bQ}_{l\widetilde{Y}})_{\rho}),
\end{equation*}
where
$(\pi_!\bar{\bQ}_{l\widetilde{Y}})_{\rho}=\Hom_{\bar{\bQ}_l[\mathbf{W}_G]}(\rho,\pi_!\bar{\bQ}_{l\widetilde{Y}})$
is an irreducible local system on $Y$. We have
\begin{equation*}
\pi_!\bar{\bQ}_{l\widetilde{Y}}=\pi''_!(\pi'_!\bar{\bQ}_{l\widetilde{Y}})
=\bigoplus_{\rho'\in
\mathbf{W}_L^{\wedge}}(\rho'\otimes\pi''_!((\pi'_!\bar{\bQ}_{l\widetilde{Y}})_{\rho'}))
\end{equation*}
hence
\begin{equation*}
\pi''_!((\pi'_!\bar{\bQ}_{l\widetilde{Y}})_{\rho'})=
\Hom_{\bar{\bQ}_l[\mathbf{W}_L]}(\rho',\pi_!\bar{\bQ}_{l\widetilde{Y}})
=\Hom_{\bar{\bQ}_l[\mathbf{W}_L]}(\rho',\bigoplus_{\rho\in
\mathbf{W}_G^{\wedge}}(\rho\otimes(\pi_!\bar{\bQ}_{l\widetilde{Y}})_{\rho})).
\end{equation*}
We see that for any $\rho'\in \mathbf{W}_L^{\wedge}$,
\begin{equation}\label{e-4}
\pi''_!((\pi'_!\bar{\bQ}_{l\widetilde{Y}})_{\rho'})=\bigoplus_{\rho\in
\mathbf{W}_G^{\wedge}}
((\pi_!\bar{\bQ}_{l\widetilde{Y}})_{\rho}\otimes\Hom_{\bar{\bQ}_l[\mathbf{W}_L]}(\rho',\rho)).
\end{equation}

\subsection{}
Recall that we have the map (see \cite{X1})
$$\varphi:X=\{(x,gB)\in\Lg\times G/B|\Ad(g^{-1})(x)\in\Lb\}\rightarrow\Lg,(x,gB)\mapsto x.$$ Let
$X_1=\{(x,gP)\in\Lg\times G/P|\Ad(g^{-1})(x)\in\Lp\}$. Then
$\varphi$ factors as
$X\xrightarrow{\varphi'}X_1\xrightarrow{\varphi''}\Lg$ where
 $\varphi'$ is $(x,gB)\mapsto(x,gP)$ and $\varphi''$ is
$(x,gP)\mapsto x$. The maps $\varphi',\varphi''$ are proper and
surjective. We have a commutative diagram $$\CD
  \widetilde{Y}  @>\pi'>> \widetilde{Y}_1 @>\pi''>> Y \\
  @V j_0 VV @V j_1 VV @V j_2 VV  \\
  X @>\varphi'>> X_1 @>\varphi''>> \Lg
\endCD$$
where $j_2$ is $x\mapsto x$, $j_0$ is $(x,gT)\mapsto(x,gB)$ (an
isomorphism of $\widetilde{Y}$ with the open subset
$\varphi^{-1}(Y)$ of $X$) and $j_1$ is $(x,gL)\mapsto (x,gP)$ (an
isomorphism onto the open subset $\varphi''^{-1}(Y)$ of $X_1$). Note
also that $\widetilde{Y}_1$ is smooth (since $\widetilde{Y}$ is
smooth). We identify $\widetilde{Y},\widetilde{Y}_1$ with open
subsets of $X,X_1$ via the maps $j_0,j_1$ respectively. Let
$$X_L=\{(x,g(B\cap L))\in\Ll\times L/(B\cap
L)|\Ad(g^{-1})(x)\in\Lb\cap\Ll\},$$ $$X''=\{(g_1,x,pB)\in G\times
\Lp\times P/B|\Ad(p^{-1})(x)\in\Lb\}.$$
 We have a
commutative diagram with cartesian squares
$$\CD
  X @<p_1<< X'' @>p_2>> X_L \\
  @V \varphi' VV @V \phi VV @V \varphi_L VV  \\
  X_1 @<p_3<< G\times\Ln_{P}\times\Ll @>p_4>> \Ll
\endCD$$
where

$p_1$ is $(g_1,x,pB)\mapsto(\Ad(g_1)(x),g_1pB)$, a principal
$P$-bundle,

$p_2$ is $(g_1,l+n,g'B)\mapsto(l,g'(B\cap L))$ with
$l\in\Ll,n\in\Ln_{P},g'\in L$, a principal $G\times\Ln_{P}$-bundle,

$p_3$ is $(g_1,n,l)\mapsto(\Ad(g_1)(l+n),g_1P)$, a principal
$P$-bundle,

$p_4$ is $(g_1,n,l)\mapsto l$, a principal $G\times\Ln_{P}$-bundle,

$\varphi_L$ is $(x,g(B\cap L))\mapsto x$,

$\phi$ is $(g_1,l+n,g'B)\mapsto(g_1,n,l)$, with
$l\in\Ll,n\in\Ln_{P},g'\in L$.

Let $\pi_L$ be the map $\widetilde{Y}_L=\{(x,gT)\in\Ll\times
L/T|\Ad(g^{-1})(x)\in\Lt_{0}\}\rightarrow Y_L, (x,gL)\mapsto x$.
Since $p_3,p_4$ are principal bundles with connected groups, we have
$p_3^*IC(X_1,\pi'_!\bar{\bQ}_{l\widetilde{Y}})=p_4^*IC(\Ll,\pi_{L!}\bar{\bQ}_{l\widetilde{Y}_L})$
(both can be identified with
$IC(G\times\Ln_P\times\Ll,p_4^*\pi_{L!}\bar{\bQ}_{l\widetilde{Y}_L})$).
From the commutative diagram above it follows that
$p_3^*\varphi'_!\bar{\bQ}_{lX}=\phi_!p_1^*\bar{\bQ}_{lX}=\phi_!p_2^*\bar{\bQ}_{lX_L}
=p_4^*\varphi_{L!}\bar{\bQ}_{lX_L}=p_4^*IC(\Ll,\pi_{L!}\bar{\bQ}_{l\widetilde{Y}_L})$
(the last equality comes from  \cite[Proposition 6.6]{X1} for $L$
instead of $G$), hence
$p_3^*\varphi'_!\bar{\bQ}_{lX}=p_3^*IC(X_1,\pi'_!\bar{\bQ}_{l\widetilde{Y}})$.
Since $p_3$ is a principal $P$-bundle we see that
$$\varphi'_!\bar{\bQ}_{lX}=IC(X_1,\pi'_!\bar{\bQ}_{l\widetilde{Y}}).$$
It follows that
$\End(\varphi'_!\bar{\bQ}_{lX})\cong\bar{\bQ}_{l}[\mathbf{W}_L]$ and
$\varphi'_!\bar{\bQ}_{lX}=\bigoplus_{\rho'\in
\mathbf{W}_L^{\wedge}}(\rho'\otimes(\varphi'_!\bar{\bQ}_{lX})_{\rho'})$
where
\begin{equation}\label{e-1}
(\varphi'_!\bar{\bQ}_{lX})_{\rho'}=IC(X_1,(\pi'_!\bar{\bQ}_{l\widetilde{Y}})_{\rho'}).
\end{equation}
Next we show that
\begin{equation}\label{e-2}
\varphi''_!((\varphi'_!\bar{\bQ}_{lX})_{\rho'})=IC(\Lg,\pi''_!((\pi'_!\bar{\bQ}_{l\widetilde{Y}})_{\rho'})),\text{
for any }\rho'\in \mathbf{W}_L^{\wedge}.
\end{equation}
From (\ref{e-1}) we see that the restriction of
$\varphi''_!((\varphi'_!\bar{\bQ}_{lX})_{\rho'})$ to $Y$ is the
local system $\pi''_!(\pi'_!(\bar{\bQ}_{l\widetilde{Y}})_{\rho'})$.
Since $\varphi''$ is proper, (\ref{e-2}) is a consequence of
(\ref{e-1}) and the following assertion:
\begin{equation}\label{e-3}
\text{For any }i>0,
\dim\supp\cH^i(\varphi''_!((\varphi'_!\bar{\bQ}_{lX})_{\rho'}))<\dim\Lg-i.
\end{equation}
We have
$\supp\cH^i(\varphi''_!((\varphi'_!\bar{\bQ}_{lX})_{\rho'}))\subset\supp\cH^i(\varphi''_!(\varphi'_!\bar{\bQ}_{lX}))
=\supp\cH^i(\varphi_!\bar{\bQ}_{lX})$, thus (\ref{e-3}) follows from
the proof of \cite[Proposition 6.6]{X1}. Hence (\ref{e-2}) is
verified. Combining (\ref{e-2}) with (\ref{e-4}), we see that for
any $\rho'\in \mathbf{W}_L^{\wedge}$,
\begin{equation}\label{e-5}
\varphi''_!((\varphi'_!\bar{\bQ}_{lX})_{\rho'})\cong\bigoplus_{\rho\in
\mathbf{W}_G^{\wedge}}
((\varphi_!\bar{\bQ}_{lX})_{\rho}\otimes\Hom_{\bar{\bQ}_{l}[\mathbf{W}_L]}(\rho',\rho)).
\end{equation}
(Recall that we have $\varphi_!\bar{\bQ}_{lX}=\bigoplus_{\rho\in
\mathbf{W}_G^{\wedge}}(\rho\otimes(\varphi_!\bar{\bQ}_{lX})_{\rho})$
and
$(\varphi_!\bar{\bQ}_{lX})_{\rho}=IC(\Lg,(\pi_!\bar{\bQ}_{l\widetilde{Y}})_{\rho})$.)

\subsection{}\label{prop-fivenumber}
Let $(\mathrm{c},\cF)\in\mathfrak{A}_\Lg$ correspond to $\rho\in
\mathbf{W}_G^{\wedge}$ and $(\mathrm{c}',\cF')\in\mathfrak{A}_\Ll$
correspond to $\rho'\in \mathbf{W}_L^{\wedge}$ under Springer
correspondence. Let $X_1^\omega=\{(x,gP)\in X_1|x\text{
nilpotent}\}$,
$$R=\{(x,gP)\in\Lg\times(G/P)|\Ad(g^{-1})(x)\in\bar{\rc'}+\Ln_P\}\subset X_1^\omega.$$
We show that
\begin{equation}\label{y-2}
\supp(\varphi'_!\bar{\bQ}_{lX})_{\rho'}\cap X_1^\omega\subset R.
\end{equation}
Let $(x,gP)\in\supp(\varphi'_!\bar{\bQ}_{lX})_{\rho'}\cap
X_1^\omega$. The isomorphism
$p_3^*\varphi'_!\bar{\bQ}_{lX}=p_4^*\varphi_{L!}\bar{\bQ}_{lX_L}$ is
compatible with the action of $\mathbf{W}_L$. Thus
$p_3^*(\varphi'_!\bar{\bQ}_{lX})_{\rho'}=p_4^*(\varphi_{L!}\bar{\bQ}_{lX_L})_{\rho'}$
and $p_3^{-1}(\supp\
(\varphi'_!\bar{\bQ}_{lX})_{\rho'})=p_4^{-1}(\supp
(\varphi_{L!}\bar{\bQ}_{lX_L})_{\rho'}).$ Hence there exists
$(g_1,n,l)\in G\times\Ln_P\times\Ll$ such that
$(x,gP)=(\Ad(g_1)(n+l),g_1P)$ and $l\in \supp
(\varphi_{L!}\bar{\bQ}_{lX_L})_{\rho'}$. Since $x$ is nilpotent,
$n+l$ is nilpotent and thus $l$ is nilpotent. Hence $l\in\bar{\rc'}$
since by \cite[Proposition 6.6]{X1} (for $L$ instead of $G$),
\begin{eqnarray}\label{i-1}
&&(\varphi_{L!}\bar{\bQ}_{lX_L})_{\rho'}|_{\cN_L}\text{ is
}IC(\bar{\rc'},\cF')[\dim\rc'-2\nu_L](\text{ extend by zero outside
}\bar{\rc'}),
\end{eqnarray}
where $\cN_L$ is the nilpotent variety of $\Ll$. We have $g=g_1p$
for some $p\in P$ and $x=\Ad(g_1)(n+l)$, hence
$\Ad(g^{-1})(x)=\Ad(p^{-1})(n+l)\in\bar{\rc'}+\Ln_P$ and $(x,gP)\in
R$. This proves (\ref{y-2}).

We have a partition $R=\cup_{\tilde{\rc}'}R_{\tilde{\rc}'}$, where
$\tilde{\rc}'$ runs over the nilpotent $L$-orbits in $\bar{\rc'}$
and $
R_{\tilde{\rc}'}=\{(x,gP)\in\Lg\times(G/P)|\Ad(g^{-1})(x)\in\tilde{\rc}'+\Ln_P\}.
$ Then $R'=R_{\rc'}$ is open in $R$. It is clear that
$p_3^{-1}(R)=p_4^{-1}(\bar{\rc'})=G\times\Ln_P\times\bar{\rc'}$ and
$p_3^{-1}(R_{\tilde{\rc}'})=p_4^{-1}(\tilde{\rc}')=G\times\Ln_P\times\tilde{\rc}'$.

Let $\tilde{\cF}'$ be the local system on $R'$ whose inverse image
under $p_3:G\times\Ln_P\times{\rc'}\rightarrow R'$ equals the
inverse image of $\cF'$ under
$p_4:G\times\Ln_P\times{\rc'}\rightarrow\rc'$. Since $p_3,p_4$ are
principal bundles with connected groups, it follows that the inverse
image of $IC(R,\tilde{\cF}')$ under
$p_3:G\times\Ln_P\times\bar{\rc'}\rightarrow R$ equals the inverse
image of $IC(\bar{\rc'},\cF')$ under
$p_4:G\times\Ln_P\times\bar{\rc'}\rightarrow\bar{\rc'}$. It follows
that
\begin{eqnarray}\label{y-1}
&&(\varphi'_!\bar{\bQ}_{lX})_{\rho'}|_{X_1^\omega}=IC(R,\tilde{\cF}')[\dim\rc'-2\nu_L](
\text{ extend by zero outside } R).
\end{eqnarray}
(Using $p_3^*$ this is reduced to (\ref{i-1}).)

For any subvariety $S$ of $X_1$, we denote
${}_S\varphi'':S\rightarrow\Lg$ the restriction of
$\varphi'':X_1\rightarrow\Lg$ to $S$.

\begin{prop}
Let $d=\nu_{G}-\frac{1}{2}\dim\rc,\ d'=\frac{1}{2}(\dim
\rc-\dim\rc')$ and $d''=\nu_{G}-\nu_{L}-d'$. The following five
numbers coincide:
\begin{enumerate}
\item[(i)] $\dim\Hom_{\bar{\bQ}_l[\mathbf{W}_L]}(\rho',\rho)$;
\item[(ii)] the multiplicity of $\cF$ in the local system
$\cL_1=\cH^{2d}(\varphi''_!(\varphi'_!\bar{\bQ}_{lX})_{\rho'})|_{\rc}$;
\item[(iii)] the multiplicity of $\cF$ in the local system
$\cL_2=\cH^{2d''}({}_{R}\varphi''_!IC(R,\tilde{\cF}'))|_{\rc}$;
\item[(iv)] the multiplicity of $\cF$ in the local system
$$\cL_3=\cH^{2d''}({}_{R'}\varphi''_!IC(R,\tilde{\cF}'))|_{\rc}=\cH^{2d''}({}_{R'}\varphi''_!\tilde{\cF}')|_{\rc};$$
\item[(v)] the multiplicity of $\cF'$ in the local system
$\cH^{2d'}f_!(\cF)$ on $\mathrm{c}'$, where
$f:\pi_{\Lp}^{-1}(\mathrm{c}')\cap\mathrm{c}\rightarrow\mathrm{c}'$
is the restriction of
$\pi_{\Lp}:\mathfrak{p}\rightarrow\mathfrak{l}$.
\end{enumerate}
\end{prop}
\begin{proof}
For $\tilde{\rho}\in \mathbf{W}_G^{\wedge}$, the multiplicity of
$\cF$ in $\cH^{2d}((\varphi_!\bar{\bQ}_{lX})_{\tilde{\rho}})|_{\rc}$
is $1$ if $\tilde{\rho}=\rho$  and is $0$ if $\tilde{\rho}\neq\rho$.
Hence it follows from (\ref{e-5}) that the numbers in (i)(ii) are
equal.

We show that $\cL_1=\cL_2$. By (\ref{y-1}), we have
$\cL_2=\cH^{2d}({}_R\varphi''_!((\varphi'_!\bar{\bQ}_{lX})_{\rho'}|_R))|_\rc$.
It suffices to show that
${}_{(X_1-R)}\varphi''_!((\varphi'_!\bar{\bQ}_{lX})_{\rho'}|_{X_1-R})|_\rc=0$.
Assume this is not true. Then there exists $(x,gP)\in X_1-R$ such
that $x\in \rc$ and $(x,gP)\in\supp\
(\varphi'_{!}\bar{\bQ}_{lX})_{\rho'}$. Since $x$ is nilpotent, this
contradicts (\ref{y-2}).

We show that $\cL_2=\cL_3$. For any $x\in\rc$ we consider the
natural exact sequence
\begin{eqnarray*}
&&H_c^{2d-1}(\varphi''^{-1}(x)\cap(R-R'),(\varphi'_!\bar{\bQ}_{lX})_{\rho'})\xrightarrow{a}
H_c^{2d}(\varphi''^{-1}(x)\cap
R',(\varphi'_!\bar{\bQ}_{lX})_{\rho'})\\
&&\rightarrow H_c^{2d}(\varphi''^{-1}(x)\cap
R,(\varphi'_!\bar{\bQ}_{lX})_{\rho'})\rightarrow
H_c^{2d}(\varphi''^{-1}(x)\cap
(R-R'),(\varphi'_!\bar{\bQ}_{lX})_{\rho'}).
\end{eqnarray*}
It is enough to show that $H_c^{2d}(\varphi''^{-1}(x)\cap
(R-R'),(\varphi'_!\bar{\bQ}_{lX})_{\rho'})=0$ and that $a=0$. By
(\ref{y-1}), we can replace
$(\varphi'_!\bar{\bQ}_{lX})_{\rho'}|_{X_1^\omega}$ by
$IC(R,\tilde{\cF}')[\dim\rc'-2\nu_L]$. It is enough to show
\begin{eqnarray}
&&H_c^{2d''}(\varphi''^{-1}(x)\cap (R-R'),IC(R,\tilde{\cF}'))=0,\label{y-3}\\
&&H_c^{2d''-1}(\varphi''^{-1}(x)\cap(R-R'),IC(R,\tilde{\cF}'))\xrightarrow{a}
H_c^{2d''}(\varphi''^{-1}(x)\cap R',IC(R,\tilde{\cF}')) \text{ is
zero}. \label{y-4}\end{eqnarray} From Proposition \ref{prop-dim}, we
see that for any $L$-orbit $\tilde{\rc}'$ in $\bar{\rc'}$,
\begin{equation}\label{y-5}
\dim(\varphi''^{-1}(x)\cap
R_{\tilde{\rc}'})\leq(\nu_G-\frac{1}{2}\dim
\rc)-(\nu_L-\frac{1}{2}\dim\tilde{\rc}').
\end{equation}

If (\ref{y-3}) is not true, then using the partition
\begin{equation}\label{y-8}
\varphi''^{-1}(x)\cap(R-R')=\bigcup_{\tilde{\rc}'\neq\rc'}(\varphi''^{-1}(x)\cap
R_{\tilde{\rc}'}),
\end{equation}
we see that $H_c^{2d''}(\varphi''^{-1}(x)\cap
R_{\tilde{\rc}'},IC(R,\tilde{\cF}'))\neq 0$ for some
$\tilde{\rc}'\neq\rc'$. Hence there exist $i,j$ such that $2d''=i+j$
and $ H^{i}_c(\varphi''^{-1}(x)\cap
R_{\tilde{\rc}'},\cH^j(IC(R,\tilde{\cF}')))\neq0$. It follows that
$i\leq2\dim(\varphi''^{-1}(x)\cap R_{\tilde{\rc}'})\leq
2\nu_G-\dim\rc-2\nu_L+\dim\tilde{\rc}'$ (we use (\ref{y-5})). The
local system $\cH^j(IC(R,\tilde{\cF}'))\neq 0$ so that
$R_{\tilde{\rc}'}\subset\supp\ \cH^j(IC(R,\tilde{\cF}'))$ and $\dim
R_{\tilde{\rc}'}<\dim R-j$. It follows that $j<\dim R-\dim
R_{\tilde{\rc}'}=\dim\rc'-\dim\tilde{\rc}'$ and $i+j<2d''$ in
contradiction to $i+j=2d''$. This proves (\ref{y-3}).

To prove (\ref{y-4}), we can assume that $\tk$ is an algebraic
closure of a finite field $\mathbf{F}_q$, that $G$ has a fixed
$\mathbf{F}_q$-structure with Frobenius map $F:G\rightarrow G$, that
$P,B,L,T$ (hence $X_1,\varphi''$) are defined over $\mathbf{F}_q$,
that any $\tilde{\rc}'$ as above is defined over $\mathbf{F}_q$,
that $F(x)=x$ and that we have an isomorphism
$F^*\cF'\rightarrow\cF'$ which makes $\cF'$ into a local system of
pure weight 0. Then we have natural (Frobenius) endomorphisms of
$H_c^{2d''-1}(\varphi''^{-1}(x)\cap(R-R'),IC(R,\tilde{\cF}'))$ and $
H_c^{2d''}(\varphi''^{-1}(x)\cap
R',IC(R,\tilde{\cF}'))=H_c^{2d''}(\varphi''^{-1}(x)\cap
R',\tilde{\cF}')$ compatible with $a$. To show that $a=0$, it is
enough to show that
\begin{eqnarray}\label{y-6}
&&H_c^{2d''}(\varphi''^{-1}(x)\cap R',IC(R,\tilde{\cF}'))\text{ is
pure of weight } 2d'';
\\&&\label{y-7}
H_c^{2d''-1}(\varphi''^{-1}(x)\cap(R-R'),IC(R,\tilde{\cF}'))\text{
is mixed of weight } \leq 2d''-1.
\end{eqnarray}
Since $\dim(\varphi''^{-1}(x)\cap R')\leq d''$ (see(\ref{y-5})),
(\ref{y-6}) is clear. Using the partition (\ref{y-8}), we see that
to prove (\ref{y-7}), it is enough to prove that
$H_c^{2d''-1}(\varphi''^{-1}(x)\cap
R_{\tilde{\rc}'},IC(R,\tilde{\cF}'))$ is mixed of weight $\leq
2d''-1$ for any ${\tilde{\rc}'}\neq\rc'$.

Using the hypercohomology spectral sequence we see that it is enough
to prove if $i,j$ are such that $2d''-1=i+j$, then
$H_c^{i}(\varphi''^{-1}(x)\cap
R_{\tilde{\rc}'},\cH^j(IC(R,\tilde{\cF}')))$ is mixed of weight
$\leq 2d''-1$ for any $\tilde{\rc}'$. By Gabber's theorem [BBD,
5.3.2], the local system $\cH^j(IC(R,\tilde{\cF}'))$ is mixed of
weight $\leq j$. Then by Deligne's theorem [BBD, 5.1.14(i)],
$H_c^{i}(\varphi''^{-1}(x)\cap
R_{\tilde{\rc}'},\cH^j(IC(R,\tilde{\cF}')))$ is mixed of weight
$\leq i+j=2d''-1$. This proves (\ref{y-7}). Hence $\cL_2=\cL_3$ is
proved.

Now consider the diagram
$\mathbf{V}\xleftarrow{f_{2}}\mathbf{V}'\xrightarrow{f_{1}}\rc$,
where $\mathbf{V}'=\varphi''^{-1}(\rc)\cap
R'=\{(x,gP)\in\rc\times(G/P)|\Ad(g^{-1})(x)\in\rc'+\Ln_P\}$,
$\mathbf{V}=P\backslash(\rc'\times G)$ with $P$ acting by
$p:(x,g)\mapsto(\Ad(\pi(p))(x),gp^{-1})$, $\pi:P\rightarrow L$ the
natural projection, $f_{2}(x,gP)=P$-orbit of
$(\pi_{\Lp}(\Ad(g^{-1})(x)),g)$, $f_{1}(x,gP)=x$. We have
$G$-actions on $\mathbf{V}$ by $g': (x,g)\mapsto(x,g'g)$, on
$\mathbf{V}'$ by $g':(x,gP)\mapsto(\Ad(g')(x),g'gP)$ and on $\rc$ by
$g':x\mapsto\Ad(g')(x)$. Then $f_{1}$ and $f_{2}$ are
$G$-equivariant and $G$ acts transitively on $\mathbf{V}$ and $\rc$.

Note all fibers of $f_{1}$ have dimension $\leq d''$ and all fibers
of $f_{2}$ have dimension $\leq d'$. Applying \cite[8.4(a)]{Lu4}
with $\mathcal{E}_1=\cF$ and with $\mathcal{E}_{2}$ the local system
on $\mathbf{V}$ whose inverse image under the natural map
$\rc'\times G\rightarrow \mathbf{V}$ is
$\cF'\boxtimes\bar{\bQ}_{l}$, we see that the numbers (iv) and (v)
are equal. This completes the proof of the proposition.
\end{proof}

\subsection{} Now we are ready to prove the restriction formula
$(\mathbf{R})$. Let the notation be as in \ref{sec-res}. Let $\rc$
be the $G$-orbit of $x$ and $\rc'$ be the $L$-orbit of $x'$. Let
$\tau: G/Z_G^0(x)\rightarrow G/Z_G(x)\simeq\rc$ be the Galois
 covering 
 with group $A_G(x)$. We have the following commutative
diagram$$\CD
  Y_{x,x'} @>a>> Y_{x,x'}/Z_G^0(x) \\
  @V b VV @V \wr VV  \\
  (x'+\Ln_P)\cap\rc @<\tau<< \tau^{-1}((x'+\Ln_P)\cap\rc),
\endCD$$
where $a$ is the natural projection and $b$ is given by $g\mapsto
\Ad(g^{-1})(x)$. Then $a$ induces an $A_G(x)$-equivariant bijection
between $S_{x,x'}$ and the set of irreducible components of
$\tau^{-1}((x'+\Ln_P)\cap\rc)$ of dimension $d'=\frac{1}{2}(\dim
\rc-\dim\rc')$ (note that $\dim (x'+\Ln_P)\cap\rc\leq d'$ by
Proposition \ref{prop-dim} (i)).

Assume $\cF$ corresponds to $\phi\in A_G(x)^\wedge$ and $\cF'$
corresponds to $\phi'\in A_L(x')^\wedge$. We have $\cF\simeq
\Hom_{A_G(x)}(\phi,\tau_*\bar{\bQ}_l)$ and thus
$H_c^{2d'}((x'+\Ln_P)\cap\rc,\cF)\cong(H_c^{2d'}((x'+\Ln_P)\cap\rc,\tau_*\bar{\bQ}_l)\otimes\phi^\wedge)^{A_G(x)}
\cong(H_c^{2d'}(\tau^{-1}((x'+\Ln_P)\cap\rc),\bar{\bQ}_l)\otimes\phi^\wedge)^{A_G(x)}.$
Then the number (v) in Proposition \ref{prop-fivenumber} is equal to
$$\la \phi',H_c^{2d'}(f^{-1}(x'),\cF)\ra_{A_L(x')}=\la
\phi',H_c^{2d'}((x'+\Ln_P)\cap\rc,\cF)\ra_{A_L(x')}=\la
\phi\otimes\phi',\varepsilon_{x,x'}\ra_{A_G(x)\times A_L(x')}.$$
Hence the restriction formula $(\mathbf{R})$ follows from
Proposition \ref{prop-fivenumber} ((i)=(v)).

\section{Orthogonal Lie algebras}\label{sec-or}
In this section we assume that $G=SO(N)$. Let $\tilde{G}=O(N)$.
\subsection{}\label{comp-or}
Let $x\in\Lg$ be nilpotent. The $\tilde{G}$-orbit $\rc$ of $x$ is
characterized by the following data (\cite{Hes}):

(d1) The sizes of the Jordan blocks of $x$ give rise to a partition
$\lambda$ of $N$,
$0\leq\lambda_1\leq\lambda_2\leq\cdots\leq\lambda_{s}$.

(d2) For each $\lambda_i$, $i=1,\ldots,s$, there is an integer
$\chi(\lambda_i)$ satisfy
$\frac{\lambda_i}{2}\leq\chi(\lambda_i)\leq\lambda_i$. Moreover,
$\chi(\lambda_i)\geq\chi(\lambda_{i-1}),\lambda_i-\chi(\lambda_i)\geq\lambda_{i-1}-\chi(\lambda_{i-1})$,
$i=2,\ldots,s$.

Denote $m(\lambda_i)$ the multiplicity of $\lambda_i$ in the
partition $\lambda$. If $N$ is even, then $m({\lambda_i})$ is even
for each $\lambda_i>0$. If $N$ is odd, then the set
$\{\lambda_i>0|m(\lambda_i)\text{ is odd}\}$ is $\{a,a-1\}\cap\{1,2,\ldots\}$ for some
integer $a\geq 1$.

We write $x$ (or $\rc$)
$=(\lambda,\chi)=(\lambda_{s})_{\chi(\lambda_s)}\cdots(\lambda_1)_{\chi(\lambda_1)}$.
The component groups
$\tilde{A}(x)=Z_{\tilde{G}}(x)/Z_{\tilde{G}}^0(x)$ and
$A(x)=Z_{G}(x)/Z_{G}^0(x)$ can be described as follows (see
\cite{X1}). Let $\epsilon_i$ correspond to $\lambda_i$,
$i=1,\ldots,s$. Then $\tilde{A}(x)$ is isomorphic to the abelian
group generated by $\{\epsilon_i,1\leq i\leq s|\chi(\lambda_i)\neq
\lambda_i/2\}$ with relations
\begin{enumerate}
\item[(r1)] $\epsilon_i^2=1$,
\item[(r2)] $\epsilon_i=\epsilon_{i+1}\text{ if
}\chi(\lambda_i)+\chi(\lambda_{i+1})>\lambda_{i+1}$,
\item[(r3)] $\epsilon_i=1$ if $m(\lambda_i)$ is odd.
\end{enumerate}
If $N$ is even, $A(x)$ is the subgroup of $\tilde{A}(x)$ consisting
of those elements that can be written as a product of even number of
generators.
\subsection{}\label{prop-1} Let $\rc'=(\lambda',\chi')\in
f_x(\tilde{\cP}_x)$, $\mathbf{Y}=f_x^{-1}(\rc')$ and
$\bX=\varrho_x^{-1}(\mathbf{Y})$ (see \ref{ss1-2}). Spaltenstein
\cite{Spal} has described the necessary and sufficient conditions
for $\dim \bX=\dim \cB_x$ as follows.
\begin{prop}[\cite{Spal}]
We have $\dim \bX=\dim \cB_x$ if and only if $(\lambda',\chi')$
satisfy (a) or (b):

\noindent (a) Assume that
$\lambda_i\neq\lambda_{i+1}\neq\lambda_{i+2}$ and
$\chi(\lambda_{i+2})=\lambda_{i+2}$. $\lambda_j'=\lambda_j$, $j\neq
i+2,i+1$,
$\lambda_{i+2}'=\lambda_{i+2}-1,\lambda_{i+1}'=\lambda_{i+1}-1$,
$\chi'(\lambda_j')=\chi(\lambda_j)$ if $j>i+2$,
$\chi'(\lambda_j')=\lambda_j'$ if $j\leq i+2$.

In this case, $\dim \mathbf{Y}=s-i-2$.

\noindent (b) Assume that $\lambda_{i+1}=\lambda_{i}>\lambda_{i-1}$.
$\lambda_j'=\lambda_j$, $j\neq i+1,i$,
$\lambda_{i+1}'=\lambda_{i+1}-1, \lambda_i'=\lambda_i-1$,
$\chi'(\lambda_j')=\chi(\lambda_j)$, $j\neq i,i+1$ and
$\chi'(\lambda_i')=\chi'(\lambda_{i+1}')\in\{\chi(\lambda_i),\chi(\lambda_i)-1\}$
satisfies $\lambda_i'/2\leq\chi'(\lambda_i')\leq\lambda_i'$,
$\chi(\lambda_{i-1})\leq\chi'(\lambda_i')\leq\chi(\lambda_{i-1})+\lambda_i-\lambda_{i-1}-1$.

In this case, $\dim \mathbf{Y}=s-i$ if
$\chi'(\lambda_i')=\chi(\lambda_i)$ and $\dim \mathbf{Y}=s-i-1$ if
$\chi'(\lambda_i')=\chi(\lambda_i)-1$.
\end{prop}

\subsection{}\label{prop-gps}From now on let $\rc'$ be as in Proposition \ref{prop-1}.
Let $\tilde{A}_P$ and $\tilde{A}_P'$ be defined as in \ref{ss1-2}.
\begin{prop} The group $Z_{\tilde{G}}(x)$ acts transitively on
$\mathbf{Y}$. The group $\tilde{A}_P$ is the subgroup of
$\tilde{A}(x)$ generated by the elements $\epsilon_i$ which appear
both in the generators of $\tilde{A}(x)$ and of $\tilde{A}'(x')$.
The group $\tilde{A}_P'$ is the smallest subgroup of
$\tilde{A}'(x')$ such that the map $\tilde{A}_P\rightarrow
\tilde{A}'(x')/\tilde{A}_P'$ given by $\epsilon_i\mapsto
\epsilon_i'$ is a morphism.
\end{prop}

\begin{coro}
 The variety $\mathbf{Y}$ has two irreducible components (and $|\tilde{A}(x)/\tilde{A}_P|=2$)
 if $\rc'$ is as in Proposition \ref{prop-1} (b) with $\chi(\lambda_i)=\frac{\lambda_i+1}{2},
\lambda_{i+2}-\chi(\lambda_{i+2})\geq\frac{\lambda_i+1}{2}$ and
$\chi'(\lambda_i')=\chi(\lambda_i)-1$.

In this case, suppose
$D=\{1,\epsilon_i\}=\{1,\epsilon_{i+1}\}\subset \tilde{A}(x)$, then
$\tilde{A}(x)=D\times \tilde{A}_P$. In the other cases, $\mathbf{Y}$
is irreducible and $\tilde{A}_P=\tilde{A}(x)$.
\end{coro}

\begin{coro}
The group $\tilde{A}_P'$ is trivial, except in the following cases
where it has order 2:

(a) $\tilde{A}_P'=\{1,\epsilon_{i+3}'\}\subset \tilde{A}'(x')$ if
$\rc'$ is as in Proposition \ref{prop-1} (a) with
$\lambda_{i+2}+\chi(\lambda_{i+3})=\lambda_{i+3}+1$.

(b) $\tilde{A}_P'=\{1,\epsilon_{i+1}'\epsilon_{i+2}'\}\subset
\tilde{A}'(x')$ if $\rc'$ is as in Proposition \ref{prop-1} (b) with
$\chi(\lambda_i)\neq\frac{\lambda_i+1}{2},
\chi(\lambda_{i+2})+\chi(\lambda_{i})=\lambda_{i+2}+1$ and
$\chi'(\lambda_i')=\chi(\lambda_i)-1$.

\end{coro}

\subsection{}\label{ssec-yc}
Assume $G=O(2n+1)=O(V)$ and $x$ corresponds to the form module
$$V=W_{l_1}(\lambda_1)\oplus\cdots\oplus W_{l_k}(\lambda_k)\oplus
D(\lambda_{k+1})\oplus W_{\lambda_{k+2}}(\lambda_{k+2})\cdots\oplus
W_{\lambda_s}(\lambda_s),$$ where $l_i=\chi(\lambda_i),
i=1,\ldots,k$. (Note that $\lambda_i$ are different from those in
\ref{comp-or}. We use here notations from \cite{X1} (hence $\lambda_1\geq\lambda_2\geq\cdots\geq\lambda_s$).) We describe
the orbits $\rc'$ and the corresponding set $\mathbf{Y}$.

We view $V$ as an $A=k[t]$-module by $\sum a_it^iv=\sum a_i(x^iv)$.
For all $i\geq 1$, let
$$W_i=\ker t\cap \text{Im}(t^{i-1}).$$ Denote $\mathbb{P}(W)$ the set of all lines in the space $W$. We identify ${\cP}_x$ with
$\mathbb{P}(\ker x\cap\alpha^{-1}(0))$ (where $\alpha$ is the
quadratic form on $V$). Let $\bY$ be as in \ref{prop-gps}. There
exists a unique $i_0$ such that
$\mathbf{Y}\subset\bP(W_{i_0})-\bP(W_{i_0+1})$. Then $i_0=\lambda_j$
for some $j=1,\ldots,s$ or $i_0=\lambda_{k+1}-1$. Write
$V'=\Sigma^\p/\Sigma$. We have the following cases.

(i) Assume $i_0=\lambda_j$, $1\leq j\leq k$,
$\lambda_{j}-1\geq\lambda_{j+1}$ and $\lambda_j-l_j-1\geq
\lambda_{j+1}-l_{j+1}$.
\begin{eqnarray*}&&\rc'=W_{l_1}(\lambda_1)\oplus\cdots\oplus
W_{l_j}(\lambda_j-1)\oplus\cdots\oplus
D(\lambda_{k+1})\oplus\cdots\oplus W_{\lambda_s}(\lambda_s),
\\
&&\mathbf{Y}=\{\tk t^{\lambda_j-1}w|t^{\lambda_j}w=0,w\notin\text{Im
} t,\alpha(t^{l_j-1}w)\neq 0\},\dim \mathbf{Y}=2j-1.
\end{eqnarray*}

(ii) Assume $i_0=\lambda_j$, $1\leq j\leq k$,
$\lambda_{j}-1\geq\lambda_{j+1}$, $l_j-1\geq l_{j+1}$ and
$l_{j}-1\geq[\lambda_{j}/2]$. Let $$\mathbf{Y}'=\{\tk
t^{\lambda_j-1}w|t^{\lambda_j}w=0,w\notin\text{Im }
t,\alpha(t^{l_j-1}w)= 0\}.$$
\begin{eqnarray*}&&\rc'=W_{l_1}(\lambda_1)\oplus\cdots\oplus
W_{l_j-1}(\lambda_j-1)\oplus\cdots\oplus
D(\lambda_{k+1})\oplus\cdots\oplus W_{\lambda_s}(\lambda_s),
\\
&&\mathbf{Y}=\mathbf{Y}'\text{  except if
}\lambda_a-l_a>\lambda_{j-1}-l_{j-1}\text{ for all
}\lambda_a>\lambda_{j-1}, \text{ and
}l_{j-1}=l_j>\frac{\lambda_{j-1}+1}{2},\\&& \qquad\quad\ \text{then
} \mathbf{Y}=\mathbf{Y}'-\{\Sigma\in
\mathbf{Y}'|\chi_{V'}(\lambda_{j-1})=l_{j-1}-1\}(\text{an open dense
subset in }\mathbf{Y}'),\\&&\dim \mathbf{Y}=2j-2.
\end{eqnarray*}

(iii) Assume $i_0=\lambda_j$, $j\geq k+2$ and
$\lambda_j\geq\lambda_{j+1}+1$.
\begin{eqnarray*}&&\rc'=W_{l_1}(\lambda_1)\oplus\cdots\oplus
D(\lambda_{k+1})\oplus \cdots\oplus
W_{\lambda_j-1}(\lambda_j-1)\oplus\cdots\oplus
W_{\lambda_s}(\lambda_s),\\
&&\mathbf{Y}=\{\tk t^{\lambda_j-1}w|t^{\lambda_j}w=0,w\notin\text{Im
} t,\alpha(t^{\lambda_j-1}w)= 0\},\dim \mathbf{Y}=2j-2.
\end{eqnarray*}

(iv) Assume $i_0=\lambda_{k+1}$ and $l_k=\lambda_k$.
{\small\begin{eqnarray*}&&\rc'=W_{l_1}(\lambda_1)\oplus \cdots\oplus
W_{l_{k-1}}(\lambda_{k-1})\oplus D(\lambda_k)\oplus
W_{\lambda_{k+1}-1}(\lambda_{k+1}-1)\oplus
W_{\lambda_{k+2}}(\lambda_{k+2})\oplus\cdots\oplus
W_{\lambda_s}(\lambda_s),
\\
&&\mathbf{Y}=\{\tk
(t^{\lambda_{k+1}-1}w+t^{\lambda_k-1}w')|t^{\lambda_{k+1}-1}w\text{
spans } V^\p, t^{\lambda_k}w'=0,w'\notin\text{Im
}t,\alpha(t^{\lambda_k-1}w')=\alpha(t^{\lambda_{k+1}-1}w)\},\\&&\dim
\mathbf{Y}=2k-1.
\end{eqnarray*}}

(v) Assume $i_0=\lambda_{k+1}-1$ and
$\lambda_{k+1}-2\geq\lambda_{k+2}$. Let $$\mathbf{Y}'=\{\tk
t^{\lambda_{k+1}-2}w|t^{\lambda_{k+1}-1}w=0,w\notin\text{Im
}t,\alpha(t^{\lambda_{k+1}-2}w)=0\}.$$
\begin{eqnarray*}&&\rc'=W_{l_1}(\lambda_1)\oplus\cdots\oplus
W_{l_k}(\lambda_k)\oplus D(\lambda_{k+1}-1)\oplus
W_{\lambda_{k+2}}(\lambda_{k+2})\oplus\cdots\oplus
W_{\lambda_s}(\lambda_s),\\
&&\mathbf{Y}=\mathbf{Y}'\text{ except if
}\lambda_a-l_a>\lambda_k-l_k\text{ for all }\lambda_a>\lambda_{k},
\text{ and }l_{k}=\lambda_{k+1}>\frac{\lambda_{k}+1}{2},\\&&
\qquad\text{ then } \mathbf{Y}=\mathbf{Y}'-\{\Sigma\in
\mathbf{Y}'|\chi_{V'}(\lambda_{k})=l_{k}-1\}(\text{an open dense
subset in }\mathbf{Y}'),\\&&\dim \mathbf{Y}=2k.
\end{eqnarray*}

\subsection{}\label{ssec-3}

Let $x,\rc'$, $\mathbf{Y},\bX$ be as in \ref{ssec-yc}. Assume
$\Sigma\in \mathbf{Y}\subset
\mathbb{P}(W_{i_0})-\mathbb{P}(W_{i_0+1})$. Let $X(\Sigma)$ be the
set of nondegenerate submodules $M$ of $V$ satisfying the following
conditions:

\smallskip

 c1) $\Sigma\subset M$ and $M$ has no proper submodule
containing $\Sigma$,

 c2) $\chi_M(i_0)=\chi_V(i_0)$. Moreover, in case $\mathrm{(iv)}$
of   \ref{ssec-yc}, $\chi_M(\lambda_k)=\chi_V(\lambda_k)$.

\smallskip

\noindent We describe the set $X(\Sigma)$ in the cases (i)-(v) of
  \ref{ssec-yc} in the following.
\begin{enumerate}
\item[(i)] Let $\Sigma=\tk v\in \mathbf{Y}$, where $v=t^{\lambda_j-1}w$. There
exists $v'\in W_{\lambda_j}-W_{\lambda_{j+1}}$, such that
$\beta(v',w)\neq 0$. Take $w'$ such that $v'=t^{\lambda_j-1}w'$.
Then $M=Aw\oplus Aw'\in X(\Sigma)$ and every module in $X(\Sigma)$
is obtained in this way. It is easily seen that
$M=W_{l_j}(\lambda_j)$.
\item[(ii)] Let $\Sigma=\tk v\in \mathbf{Y}$, where $v=t^{\lambda_j-1}w$. There exist
$v'=t^{\lambda_j-1}w'\in W_{\lambda_j}-W_{\lambda_{j+1}}$, such that
$\beta(v',w)\neq 0$ and $\alpha(t^{l_j-1}w')\neq 0$. Then
$M=Aw\oplus Aw'\in X(\Sigma)$ and every module in $X(\Sigma)$ is
obtained in this way. It is easily seen that $M=W_{l_j}(\lambda_j)$.
\item[(iii)] Let $\Sigma=\tk v\in \mathbf{Y}$, where $v=t^{\lambda_j-1}w$. There
exists $v'=t^{\lambda_j-1}w'\in W_{\lambda_j}-W_{\lambda_{j+1}}$,
such that $\beta(v',w)\neq 0$ and $\alpha(t^{\lambda_j-1}w')\neq 0$.
Then $M=Aw\oplus Aw'\in X(\Sigma)$ and every module in $X(\Sigma)$
is obtained in this way. It is easily seen that
$M=W_{\lambda_j}(\lambda_j)$.
\item[(iv)] Let $\Sigma=\tk v\in \mathbf{Y}$, where
$v=t^{\lambda_{k+1}-1}w+t^{\lambda_k-1}w'$. There exists
$v_1=t^{\lambda_{k+1}-2}w_1\in
W_{\lambda_{k+1}-1}-W_{\lambda_{k+1}}$ such that $\beta(w,v_1)\neq
0$ and $v_1'=t^{\lambda_k-1}w_1'\in
W_{\lambda_{k}}-W_{\lambda_{k}+1}$ such that
$\beta(v_1',t^{\lambda_k-1}w')\neq 0$. Then $M=Aw\oplus Aw_1\oplus
Aw'\oplus Aw_1'\in X(\Sigma)$ and every module in $X(\Sigma)$ is
obtained in this way. It is easily seen that
$M=W_{\lambda_k}(\lambda_k)\oplus D(\lambda_{k+1})$.
\item[(v)] Let $\Sigma=\tk v\in \mathbf{Y}$, where $v=t^{\lambda_{k+1}-2}w$. There
exists $v'=t^{\lambda_{k+1}-1}w'\in
W_{\lambda_{k+1}}-W_{\lambda_{k+1}+1}$ such that $\beta(w',v)\neq
0$. Then $M=Aw\oplus Aw'\in X(\Sigma)$ and every module in
$X(\Sigma)$ is obtained in this way. It is easily seen that
$M=D(\lambda_{k+1})$.
\end{enumerate}
\subsection{}
Let $M\in X(\Sigma)$ and $M^\p=\{v\in V|\beta(v,M)=0\}$. Then $M^\p$
is a non-degenerate submodule of $V$. In cases (i)-(iii) of
\ref{ssec-yc}, we have that $V=M\oplus M^\p$. In cases (iv)-(v), we
have $V=M+M^\p$ and $M\cap M^\p=V^\p$. The nondegenerate submodule
$M^\p$ has orthogonal decomposition $M^\p=M'\oplus D(1)$, where $M'$
is a non-defective submodule. Hence $V=M'\oplus M$ (direct sum of
orthogonal submodules). Now the map $t':\Sigma^\p/\Sigma\rightarrow
\Sigma^\p/\Sigma$ induced by $t$ is given by the form module
$\frac{\Sigma^\p\cap M}{\Sigma}\oplus M'$, where $M'$ is defined as
above in cases (iv)-(v) and $M'=M^\p$ in cases (i)-(iii). We write
$\tilde{M}=\frac{\Sigma^\p\cap M}{\Sigma}$.

We explain case (ii) of \ref{ssec-yc} in detail and the other cases
are similar. In this case
$M=W_{l_j}(\lambda_j),\tilde{M}=W_{l_j-1}(\lambda_j-1)$. Recall that
$\chi_{W_{l_j}}(\lambda_j)=[\lambda_j:l_j]$, where
$[m:l]:\mathbb{N}\rightarrow\mathbb{N}$ is defined by
$[m;l](k)=\max\{0,\min\{k-m+l,l\}\}.$ We have that
$\chi(\lambda_i)=\max\{\chi_{M'}(\lambda_i),\chi_{M}(\lambda_i)\}$
and
$\chi'(\lambda_i)=\max\{\chi_{M'}(\lambda_i),\chi_{\tilde{M}}(\lambda_i)\}$.
One easily check that $\chi(\lambda_i)=\chi'(\lambda_i)$ for $i\geq
j+1$,  $\chi(\lambda_j-1)=\chi'(\lambda_j-1)=l_j-1$, and
$\chi(\lambda_i)=l_i=\max\{\chi_{M'}(\lambda_i),l_j\},
\chi'(\lambda_i)=\max\{\chi_{M'}(\lambda_i),l_j-1\}$ for $i\leq
j-1$.

If $l_{j-1}>l_j$, then $l_{i}>l_j$, $\forall\ i\leq j-1$. It follows
that $\chi_{M'}(\lambda_i)=l_i$ and thus $\chi'(\lambda_i)=l_i$,
$\forall\ i\leq j-1$. Assume $l_{j-1}=l_j$ and there exists some
$\lambda_i>\lambda_{j-1}$ such that
$\lambda_i-l_i=\lambda_{j-1}-l_{j-1}$, then $l_i>l_j$ and
$\chi_{M'}(\lambda_i)=l_i$. It follows that
$\chi_{M'}(\lambda_{j-1})=\lambda_{j-1}-\lambda_i+l_i=l_{j-1}$ and
thus $\chi_{M'}(\lambda_i)=l_i,\forall\ i\leq j-1$. Assume
$l_{j-1}=l_j\geq[(\lambda_{j-1}+1)/2]+1$ and for all
$\lambda_i>\lambda_{j-1}$, $\lambda_i-l_i>\lambda_{j-1}-l_{j-1}$.
Since we require $\chi_{V'}(\lambda_{j-1})=l_{j-1}$,
$\chi_{M'}(\lambda_{j-1})=l_{j-1}$ and thus
$\chi_{M'}(\lambda_i)=l_i,\forall\ i\leq j-1$. In any case,
$\chi'(\lambda_i)=l_i$, $\forall\ i\leq j-1$. Hence $\rc'$ is of the
form as stated.

\subsection{}\label{ssec-2} The form modules $(\Sigma^\p\cap M)/\Sigma$ are described in the
following.
\begin{enumerate}
\item Assume $x=W_{m}(2m)$, $m\geq 1$. Then
$\tilde{\cP}_x=\mathbb{P}(\ker x)$ and
$f_x(\tilde{\cP}_x)=W_{m}(2m-1)$.

\item Assume $x=W_{m+1}(2m+1)$, $m\geq 1$. Then
$\tilde{\cP}_x=\mathbb{P}(\ker x)$, $\mathbf{Y}_1=f_x^{-1}(W_m(2m))$
consists of two points and
$\mathbf{Y}_2=f_x^{-1}(W_{m+1}(2m))=\tilde{\cP}_x-\mathbf{Y}_1$.

\item Assume $x=W_l(m)$, $(m+1)/2<l<m$. Then
$\tilde{\cP}_x=\mathbb{P}(\ker x)$,
$\mathbf{Y}_1=f_x^{-1}(W_{l-1}(m-1))$ consists of one point and
$\mathbf{Y}_2=f_x^{-1}(W_l(m-1))=\tilde{\cP}_x-\mathbf{Y}_1$.

\item Assume $x=W_m(m)$, $m\geq 2$. Then $\tilde{\cP}_x$ consists one
point and $f_x(\tilde{\cP}_x)=W_{m-1}(m-1)$.

\item Assume $x=W_1(1)$. Then $\tilde{\cP}_x$ consists of two points
and $f_x(\tilde{\cP}_x)=\{0\}$.

\item Assume $x=W_m(m)\oplus D(k)$, $m\geq k\geq 1$. Then
$\tilde{\cP}_x=\mathbb{P}(\ker x\cap\alpha^{-1}(0))$ and
$f_x^{-1}(D(m)\oplus W_{k-1}(k-1))=\mathbb{P}((W_{k}-W_{k+1})\cap
\alpha^{-1}(0))$.

\item Assume $x=D(m)$, $m\geq 2$. Then $\tilde{\cP}_x$ consists of
one point and $f_x(\tilde{\cP}_x)=D(m-1)$.
\end{enumerate}

\subsection{}
We prove Proposition \ref{prop-gps} for $O(2n+1)$. The proof for
$O(2n)$ is entirely similar and simpler. We use similar ideas as in
\cite{Spal2}. We first show that $Z_G(x)$ acts transitively on
$\mathbf{Y}$. Consider
$\widetilde{\mathbf{Y}}^*=\{(\Sigma,M)|\Sigma\in \mathbf{Y},M\in
X(\Sigma)^*\}$, where $X(\Sigma)^*$ is the nonempty subset $\{M\in
X(\Sigma)|\chi_{M^\p}(\lambda_a)=\chi_{V}(\lambda_a),\forall a\neq
j\text{ in cases (i)-(iii)}, a\neq k,k+1\text{ in case (iv) and
}a\neq k+1\text{ in case (v)}\}$  in $X(\Sigma)$. For $M\in
pr_2(\widetilde{\mathbf{Y}}^*)$, the equivalence classes of $M,M^\p$
do not depend on the choice of $\Sigma\in \mathbf{Y}$ such that
$(\Sigma,M)\in\widetilde{\mathbf{Y}}^*$. It follows that $Z_G(x)$
acts transitively on $pr_2(\widetilde{\mathbf{Y}}^*)$.

Fix $\Sigma\in \mathbf{Y}$ and $M\in X(\Sigma)^*$. Let $Z_M$ be the
stabilizer of $M$ in $Z_G(x)$. The quadratic form $\alpha$ on $V$
restricts to nondegenerate quadratic forms on $M$, $M^\p$ (or $M'$,
if $M'\neq M^\p$). Let $G(M)$, $G(M^\p)$ (or $G(M')$) be the groups
preserving the respective quadratic forms and $\Lg(M)$, $\Lg(M^\p)$
(or $\Lg(M')$) the Lie algebras. Let $x_M$, $x_{M^\p}$ (or $x_{M'}$)
be the restriction of $x$ on $M$, $M^\p$ (or $M'$) respectively.
Then $x_M\in\Lg(M)$, $x_{M^\p}\in\Lg(M^\p)$ (or $x_{M'}\in\Lg(M')$).
We have that $Z_M$ is isomorphic to $Z_{G(M)}(x_M)\times
Z_{G(M^\p)}(x_{M^\p})$. Set
$\widetilde{\mathbf{Y}}_M^*=pr_2^{-1}(M)=\{\Sigma\in \mathbf{Y}|M\in
X(\Sigma)^*\}$. By examining the cases (1)-(7) from \ref{ssec-2} we
see that $Z_M$ acts transitively on $\widetilde{\mathbf{Y}}_M^*$.
Thus $Z_G(x)$ acts transitively on $\widetilde{\mathbf{Y}}^*$ and
hence acts transitively on
$\mathbf{Y}=pr_1(\widetilde{\mathbf{Y}}^*)$.

Let $Z_\Sigma$ be the stabilizer of $\Sigma$ in $Z_G(x)$. The
morphism ${A}_P\rightarrow {A}'(x')/A'_P$ is induced by the natural
morphism $Z_\Sigma/Z_\Sigma^0\rightarrow A'(x')$. Since
$X(\Sigma)^*$ is irreducible, $Z_{\Sigma,M}=Z_\Sigma\cap Z_M$ meets
all the irreducible components of $Z_\Sigma$. Thus to study the
morphism $A_P\rightarrow A'(x')/A'_P$, it suffices to study the
natural morphism $Z_{\Sigma,M}/Z_{\Sigma,M}^0\rightarrow A'(x')$.

Let $x_{\tilde{M}}$ be the endomorphism of $\tilde{M}=(\Sigma^\p\cap
M)/\Sigma$ induced by $x_M$. Then $x_{\tilde{M}}\in \Lg(\tilde{M})$.
Let
$A'(x_{\tilde{M}})=Z_{G(\tilde{M})}(x_{\tilde{M}})/Z_{G(\tilde{M})}^0(x_{\tilde{M}})$.
Let $Z=\{z\in Z_{G(M)}(x_M)|z\Sigma=\Sigma\}$. We have a natural
isomorphism $Z_{\Sigma,M}\cong Z\times Z_{G(M^\p)}({x_{M^\p}})$ and
$Z_{\Sigma,M}/Z_{\Sigma,M}^0\cong Z/Z^0\times A(x_{M^\p})$. The
morphism $A(x_{M^\p})\rightarrow A'(x')$ is the one obtained as
follows. Note that $A(x_{M^\p})$ is naturally isomorphic to
$A(x_{M'})$. The system of generators of $A(x)$ is the union of the
generators of $A(x_M)$ and $A(x_{M^\p})$ and the morphism
$A(x_M)\times A(x_{M^\p})\rightarrow A(x)$ is equal to the one
induced by $Z_{G(M)}(x_M)\times Z_{G(M^\p)}(x_{M^\p})\cong
Z_M\subset Z_G(x)$. On the other hand, we have a morphism
${A}'(x_{\tilde{M}})\times A(x_{M'})\rightarrow A'(x')$ which comes
from the isomorphism $\Sigma^\p/\Sigma\cong \tilde{M}\oplus M'$ and
it is given by the system of generators. Hence the map
$A(x_{M^\p})\hookrightarrow Z_{\Sigma,M}/Z_{\Sigma,M}^0\to A'(x')$
is given by generators. It remains to identify the morphism
$Z/Z^0\rightarrow A'(x')$.

We can show by explicit calculation on the cases (1)-(7) in
  \ref{ssec-2} that the natural morphism $Z/Z^0\rightarrow
A(x_M)$ is injective and the image is generated by
$\{\epsilon_j|\lambda_j'\neq\lambda_j, \epsilon_j$ belongs to the
system of generators of ${A}(x_M)$ and $A'(x_{\tilde{M}})\}$. Using
this description of $Z/Z^0$ and the above description of the
morphism ${A}'(x_{\tilde{M}})\times A(x_{M'})\rightarrow A'(x')$, we
see that the morphism $Z/Z^0\rightarrow A'(x')$ is given by the
system of generators. So we have obtained a complete description of
the morphism $Z_{\Sigma,M}/Z_{\Sigma,M}^0\rightarrow A'(x')$ and we
deduce easily that $A_P'$ and the homomorphism $A_P\rightarrow
A'(x')/A'_P$ are as in Proposition \ref{prop-gps}.

\section{dual of symplectic Lie algebras}\label{sec-dsp}

Assume that $G=Sp(V)$ in this section.

\subsection{}\label{comp-sym-d}Let $\xi\in\Lg^*$ be nilpotent and $\alpha_\xi$, $T_\xi$ defined for $\xi$ as in \ref{ssec-spd}. The $G$-orbit
$\rc$ of $\xi$ is characterized by the following data (\cite{X2}):

(d1) The sizes of the Jordan blocks of $T_\xi$ give rise to a
partition of $2n$. We write it as
$\lambda_1\leq\lambda_2\leq\cdots\leq\lambda_{2s+1}$, where
$\lambda_1=0$.

(d2) For each $\lambda_i$, there is an integer $\chi(\lambda_i)$
satisfy $\frac{\lambda_i-1}{2}\leq\chi(\lambda_i)\leq\lambda_i$.
Moreover,
$\chi(\lambda_i)\geq\chi(\lambda_{i-1}),\lambda_i-\chi(\lambda_i)\geq\lambda_{i-1}-\chi(\lambda_{i-1})$,
$i=2,\ldots,2s+1$.

Then $m(\lambda_i)$ is even for each $\lambda_i>0$. We write $\xi$
(or $\rc$)
$=(\lambda,\chi)=(\lambda_{2s+1})_{\chi(\lambda_{2s+1})}\cdots(\lambda_1)_{\chi(\lambda_1)}$.
The component group $A(\xi)=Z_G(\xi)/Z_G^0(\xi)$ can be described as
follows (\cite{X2}). Let $\epsilon_i$ correspond to $\lambda_i$.
Then $A(\xi)$ is isomorphic to the abelian group generated by
$\{\epsilon_i|\chi(\lambda_i)\neq (\lambda_i-1)/2\}$ with relations
\begin{enumerate}
\item[(r1)] $\epsilon_i^2=1$,
\item[(r2)] $\epsilon_i=\epsilon_{i+1}\text{ if
}\chi(\lambda_i)+\chi(\lambda_{i+1})\geq\lambda_{i+1}$,
\item[(r3)] $\epsilon_i=1$, if $\lambda_i=0$.
\end{enumerate}

\subsection{}\label{lem-p-1} Let $P$ be the stabilizer of a line $\Sigma=\{\tk
v\}\subset V$ in $G$.

\begin{lem}
$\xi\in\Lp^*$ if and only if $\alpha_\xi(v)=0$ and $T_\xi(v)=0$.
\end{lem}
\begin{proof}
$P$ is the stabilizer of the flag $\{0\subset\{\tk v\}\subset\{\tk
v\}^\p\subset V\}$. Write $v_1=v$. There exists vectors $v_i$,
$i=2,\ldots,2n$ such that $v_i$, $i=1,\ldots,2n$ span $V$ and
$\beta(v_i,v_j)=\delta_{j,i+n}$, $i\leq j$. Let $x\in\Ln_P$. We have
$xv_1=0$, $xv_i=a_iv_1$,$i\neq 1,n+1$ and
$xv_{n+1}=bv_1+\sum_{i=2}^na_{n+i}v_i+\sum_{i=2}^{n}a_iv_{n+i}$.
Assume $\xi(x')=\tr(Xx')$ for any $x'\in\Lg$. A straightforward
calculation shows that $\tr(Xx)=\sum_{i=2}^n
a_i\beta_\xi(v_1,v_{n+i})+\sum_{i=2}^{n}
a_{n+i}\beta_\xi(v_1,v_{i})+b\alpha_\xi(v_1)$. Moreover,
$T_\xi(v_1)=\sum_{j=1}^n\beta_\xi(v_1,v_{n+j})v_j+\sum_{j=1}^n\beta_\xi(v_1,v_j)v_{n+j}$.

We have $\xi\in\Lp^*$ if and only if $\xi(x)=0$ for any $x\in \Ln_P$
if and only if
$\beta_\xi(v_1,v_i)=\beta_\xi(v_1,v_{n+i})=0$,$i=2,\ldots,n$ and
$\alpha_\xi(v_1)=0$. Thus $\xi\in\Lp^*$ if and only if
$\alpha_\xi(v_1)=0$ and $T_\xi(v_1)=av_1$ for some $a\in k$. Since
$T_\xi$ is nilpotent, $T_\xi(v_1)=av_1$ if and only if $a=0$. The
lemma is proved.
\end{proof}

\subsection{}\label{prop-dsp}Assume $\rc'=(\lambda',\chi')\in f_\xi(\cP_\xi)$,
$\mathbf{Y}=f_\xi^{-1}(\rc')$ and $\bX=\varrho_\xi^{-1}(\mathbf{Y})$
(see \ref{ssec-gps-d}).

\begin{prop}
We have $\dim \bX=\dim \cB_\xi$ if and only if $(\lambda',\chi')$
satisfies:

Assume $\lambda_{i+1}=\lambda_{i}>\lambda_{i-1}$.
$\lambda_j'=\lambda_j$, $j\neq i+1,i$,
$\lambda_{i+1}'=\lambda_{i+1}-1, \lambda_i'=\lambda_i-1$,
$\chi'(\lambda_j')=\chi(\lambda_j)$, $j\neq i,i+1$ and
$\chi'(\lambda_i')=\chi'(\lambda_{i+1}')\in\{\chi(\lambda_i),\chi(\lambda_i)-1\}$
satisfies $[\lambda_i'/2]\leq\chi'(\lambda_i')\leq\lambda_i'$,
$\chi(\lambda_{i-1})\leq\chi'(\lambda_i')\leq\chi(\lambda_{i-1})+\lambda_i-\lambda_{i-1}-1$.

We have $\dim \mathbf{Y}=2s-i+1$ if
$\chi'(\lambda_i')=\chi(\lambda_i)$ and $\dim \mathbf{Y}=2s-i$ if
$\chi'(\lambda_i')=\chi(\lambda_i)-1$.
\end{prop}
\subsection{}\label{prop-gps-d1} From now on let $\rc'$ be as in Proposition \ref{prop-dsp}.
Let $A_P$ and $A_{P}'$ be as in \ref{ssec-gps-d}.
\begin{prop}
The group $Z_G(\xi)$ acts transitively on $\mathbf{Y}$. The group
$A_P$ is the subgroup of $A(\xi)$ generated by the elements
$\epsilon_i$ which appear both in the generators of $A(\xi)$ and of
$A'(\xi')$. The group $A_P'$ is the smallest subgroup of $A'(\xi')$
such that the map $A_P\rightarrow A'(\xi')/A_P'$ given by
$\epsilon_i\mapsto \epsilon_i'$ is a morphism.
\end{prop}

\begin{coro}
The variety $\mathbf{Y}$ has two irreducible components (and
$|A(\xi)/A_P|=2$) if $\rc'$ is as in Proposition with
$\chi(\lambda_i)=\frac{\lambda_i}{2},
\lambda_{i+2}-\chi(\lambda_{i+2})>\lambda_{i}/2$ and
$\chi'(\lambda_i')=\chi(\lambda_i)-1$.

In this case, suppose
$D=\{1,\epsilon_i\}=\{1,\epsilon_{i+1}\}\subset {A}(\xi)$, then
${A}(\xi)=D\times {A}_P$. In the other cases, $\mathbf{Y}$ is
irreducible and ${A}_P={A}(\xi)$.
\end{coro}

\begin{coro}
 The group $A_P'$ is trivial,
except if  $\rc'$ is as in Proposition with
$\chi(\lambda_i)\neq\frac{\lambda_i}{2},
\chi(\lambda_{i+2})+\chi(\lambda_{i})=\lambda_{i+2}$ and
$\chi'(\lambda_i')=\chi(\lambda_i)-1$.

In this case, we have
${A}_P'=\{1,\epsilon_{i+1}'\epsilon_{i+2}'\}\subset {A}'(\xi')$.
\end{coro}
\subsection{}
Propositions \ref{prop-dsp} and \ref{prop-gps-d1} are proved
entirely similarly as in the orthogonal Lie algebra case.  We
describe the orbits $\rc'$ and the varieties $\mathbf{Y}$. The
detail is omitted. Assume $\xi$ corresponds to the form module
$V={}^*W_{l_1}(\lambda_1)\oplus\cdots\oplus {}^*W_{l_s}(\lambda_s)$,
where $l_i=\chi(\lambda_i)$. (Notations are as in \cite{X2}.)

We regard $V$ as an $A=\tk[t]$-module by $\sum a_it^iv=\sum a_i
T_\xi^iv$. By Lemma \ref{lem-p-1}, we can identify $\cP_\xi$ with
$\mathbb{P}(W)$, where $W=\{v\in\ker t|\alpha_\xi(v)=0\}$. Let
$\Sigma=\tk v\in \mathbf{Y}$ and $\Sigma^\p=\{v'\in
V|\beta(v',\Sigma)=0\}$. The quadratic form $\alpha_\xi$ induces a
well-defined quadratic form
$\bar{\alpha}_\xi:\Sigma^\p/\Sigma\rightarrow \Sigma^\p/\Sigma$
(note that $\beta_\xi(\Sigma^\p,\Sigma)=0$) and $T_\xi$ induces a
linear map $\bar{T}_\xi:\Sigma^\p/\Sigma\rightarrow
\Sigma^\p/\Sigma$. Then $\bar{\alpha}_\xi$ defines an element
$\xi'\in\mathfrak{sp}(\Sigma^\p/\Sigma)^*={\Ll}^{*}$. Moreover,
$\xi'\in\rc'$, $\alpha_{\xi'}=\bar{\alpha}_\xi$ and
$T_{\xi'}=\bar{T}_\xi$. We have the following cases.

(i) Assume $1\leq j\leq s$, $\lambda_{j}-1\geq\lambda_{j+1}$ and
$\lambda_j-l_j-1\geq \lambda_{j+1}-l_{j+1}$.
\begin{eqnarray*}&&\rc'={{}^*W}_{l_1}(\lambda_1)\oplus\cdots\oplus
{{}^*W}_{l_j}(\lambda_j-1)\oplus\cdots\oplus {{}^*W}_{\lambda_s}(\lambda_s),\\
&&\mathbf{Y}=\{\tk t^{\lambda_j-1}w|t^{\lambda_j}w=0,w\notin\text{Im
}t,\alpha_\xi(t^{l_j-1}w)\neq0\},\ \dim \mathbf{Y}=2j-1.
\end{eqnarray*}

(ii) Assume $1\leq j\leq s$, $\lambda_{j}-1\geq\lambda_{j+1}$,
$l_j-1\geq l_{j+1}$ and $l_{j}-1\geq[(\lambda_{j}-1)/2]$. Let
$$\mathbf{Y}'=\{\tk t^{\lambda_j-1}w|t^{\lambda_j}w=0,w\notin\text{Im
}t,\alpha_\xi(t^{l_j-1}w)=0\}.$$
\begin{eqnarray*}&&\rc'={{}^*W}_{l_1}(\lambda_1)\oplus\cdots\oplus
{{}^*W}_{l_j-1}(\lambda_j-1)\oplus\cdots\oplus {{}^*W}_{l_s}(\lambda_s), \\
&&\mathbf{Y}=\mathbf{Y}'\text{ {except if}
}\lambda_a-l_a>\lambda_{j-1}-l_{j-1}\text{ for all
}\lambda_a>\lambda_{j-1}, \text{ and
}l_{j-1}=l_j>\frac{\lambda_{j-1}}{2},\\&& \qquad\text{ then }
\mathbf{Y}=\mathbf{Y}'-\{\Sigma\in
\mathbf{Y}'|\chi_{V'}(\lambda_{j-1})=l_{j-1}-1\}(\text{an open dense
subset in }\mathbf{Y}'),\\&&\dim \mathbf{Y}=2j-2.
\end{eqnarray*}

\section{dual of odd orthogonal Lie algebras}\label{sec-dor}

Assume that $G=O(2n+1)$ in this section.

\subsection{}\label{ssec-dor}
Let $\xi\in\Lg^*$ be nilpotent. Let $V=V_{2m+1}\oplus W$ be a normal
form of $\xi$, $\beta_\xi$ and $T_\xi:W\to W$ defined for $\xi$ as
in \ref{ssec-or}. The orbit $\rc$ of $\xi$ is characterized by the
following data (\cite{X2}):

(d1) An integer $0\leq m\leq n$.

(d2) The sizes of the Jordan blocks of $T_\xi$ give rise to a
partition of $2n-2m$. We write it as
$\lambda_1\leq\lambda_2\leq\cdots\leq\lambda_{2s}$.

(d3) For each $\lambda_i$, there is an integer $\chi(\lambda_i)$
satisfy $\frac{\lambda_i}{2}\leq\chi(\lambda_i)\leq\lambda_i$.
Moreover,
$\chi(\lambda_i)\geq\chi(\lambda_{i-1}),\lambda_i-\chi(\lambda_i)\geq\lambda_{i-1}-\chi(\lambda_{i-1})$,
$i=2,\ldots,2s$.

(d4) $m\geq\lambda_{2s}-\chi(\lambda_{2s})$.

Then $m(\lambda_i)$ is even for each $\lambda_i>0$. We write $\xi$
(or $\rc$)
$=(m;\lambda,\chi)=(m;(\lambda_{2s})_{\chi(\lambda_{2s})}\cdots(\lambda_1)_{\chi(\lambda_1)})$.
The component group $A(\xi)=Z_G(\xi)/Z_G^0(\xi)$ can be described as
follows (\cite{X2}). Let $\epsilon_i$ correspond to $\lambda_i$,
$i=1\ldots,2s$. Then $A(\xi)$ is isomorphic to the abelian group
generated by $\{\epsilon_i|\chi(\lambda_i)\neq \lambda_i/2\}$ with
relations
\begin{enumerate}
\item[(r1)] $\epsilon_i^2=1$,
\item[(r2)] $\epsilon_i=\epsilon_{i+1}\text{ if
}\chi(\lambda_i)+\chi(\lambda_{i+1})>\lambda_{i+1}$,
\item[(r3)] $\epsilon_{2s}=1$ if $\chi(\lambda_{2s})\geq m$.
\end{enumerate}

\subsection{}\label{lem-p-2}  Let $P$ be the stabilizer of a line $\Sigma=\{\tk
v\}\subset V$ in $G$, where $\alpha(v)=0$.
\begin{lem}
$\xi\in\Lp^*$ if and only if $\beta_\xi(v,v')=0$ for any $v'\in V$.
\end{lem}
\begin{proof}
$P$ is the stabilizer of the flag $\{0\subset\{\tk v\}\subset\{\tk
v\}^\p\subset V\}$. Write $v_1=v$. There exists vectors $v_i$,
$i=2,\ldots,2n+1$ such that $v_i$, $i=1,\ldots,2n+1$ span $V$ and
$\beta(v_i,v_j)=\delta_{j,i+n}$, $1\leq i\leq j\leq 2n$,
$\beta(v_i,v_{2n+1})=0$, $i=1,\ldots,2n+1$, $\alpha(v_i)=0$,
$i=1,\ldots,2n$, $\alpha(v_{2n+1})=1$. Let $x\in\Ln_P$. We have
$xv_1=0$, $xv_i=a_iv_1$,$i\neq 1,n+1,2n+1$ and
$xv_{n+1}=\sum_{i=2}^na_{n+i}v_i+\sum_{i=2}^{n}a_iv_{n+i}+bv_{2n+1}$,
$xv_{2n+1}=0$. Assume $\xi(x')=\tr(Xx')$ for any $x'\in\Lg$. A
straightforward calculation shows that  $\tr(Xx)=\sum_{i=2}^n
a_i\beta_\xi(v_1,v_{n+i})+\sum_{i=2}^{n}
a_{n+i}\beta_\xi(v_1,v_{i})+b\beta_\xi(v_1,v_{2n+1})$. Thus if
$\xi\in\Lp^*$ then $\beta_\xi(v_1,v_i)=0,i\neq n+1$.

Now let $W$ be the subspace of $V$ spanned by $v_i,i=1,\ldots,2n$.
Then $\beta$ is nondegenerate on $W$. We define a map
$T:W\rightarrow W$ by $\beta(Tw,w)=\beta_\xi(w,w')$, $w,w'\in W$.
Then similar argument as in \cite[Lemma 3.11]{X2}  shows that $T$ is
nilpotent. One easily shows that
$Tv_1=\sum_{j=1}^n\beta_\xi(v_1,v_{n+j})v_j+\sum_{j=1}^n\beta_\xi(v_1,v_{j})v_{n+j}$.
It follows that $Tv_1=\beta_\xi(v_1,v_{n+1})v_1$ and thus
$\beta_\xi(v_1,v_{n+1})=0$. The lemma follows.
\end{proof}

\subsection{}\label{prop-1-d}
Let $\rc'=(m';\lambda',\chi')\in f_\xi(\cP_\xi)$,
$\mathbf{Y}=f_\xi^{-1}(\rc')$ and $\bX=\rho_\xi^{-1}(\mathbf{Y})$
(see \ref{ssec-gps-d}).

\begin{prop}
We have $\dim \bX=\dim \cB_\xi$ if and only if $(\lambda',\chi')$
and $m'$ satisfy (a) or (b):

(a) Assume $m-1\geq \lambda_{2s}-\chi(\lambda_{2s})$. $m'=m-1$,
$\lambda_i'=\lambda_i$ and $\chi'(\lambda_i')=\chi(\lambda_i)$,
$i=1,\ldots,2s$. We have $\dim \mathbf{Y}=0$;

(b) Assume that $\lambda_{i+1}=\lambda_{i}>\lambda_{i-1}$. $m'=m$,
$\lambda_j'=\lambda_j$, $j\neq i+1,i$,
$\lambda_{i+1}'=\lambda_{i+1}-1, \lambda_i'=\lambda_i-1$,
$\chi'(\lambda_j')=\chi(\lambda_j)$, $j\neq i,i+1$ and
$\chi'(\lambda_i')=\chi'(\lambda_{i+1}')\in\{\chi(\lambda_i),\chi(\lambda_i)-1\}$
satisfies $\lambda_i'/2\leq\chi'(\lambda_i')\leq\lambda_i'$,
$\chi(\lambda_{i-1})\leq\chi'(\lambda_i')\leq\chi(\lambda_{i-1})+\lambda_i-\lambda_{i-1}-1$.
We have $\dim \mathbf{Y}=2s-i+1$ if
$\chi'(\lambda_i')=\chi(\lambda_i)$ and $\dim \mathbf{Y}=2s-i$ if
$\chi'(\lambda_i')=\chi(\lambda_i)-1$.
\end{prop}
\subsection{}\label{prop-gps-d2} From now on let $\rc'$ be as in Proposition
\ref{prop-1-d}. Let $A_P$ and $A_P'$ be defined as in
\ref{ssec-gps-d}.
\begin{prop}
The group $Z_G(\xi)$ acts transitively on $\mathbf{Y}$. The group
$A_P$ is the subgroup of $A(\xi)$ generated by the elements
$\epsilon_i$ which appear both in the generators of $A(\xi)$ and of
$A'(\xi')$. The group $A_P'$ is the smallest subgroup of $A'(\xi')$
such that the map $A_P\rightarrow A'(\xi')/A_P'$ given by
$\epsilon_i\mapsto \epsilon_i'$ is a morphism.
\end{prop}

\begin{coro}
The variety $\mathbf{Y}$ has two irreducible components (and
$|A(\xi)/A_P|=2$) if $\rc'$ is as in Proposition \ref{prop-1-d} (b)
with $\chi(\lambda_i)=\frac{\lambda_i+1}{2}$,
$\chi'(\lambda_i')=\chi(\lambda_i)-1$, and
$\lambda_{i+2}-\chi(\lambda_{i+2})\geq(\lambda_{i}+1)/2$ if
$i<2s-1$, $m\geq(\lambda_{i}+1)/2$ if $i=2s-1$.

In this case, suppose
$D=\{1,\epsilon_i\}=\{1,\epsilon_{i+1}\}\subset {A}(\xi)$, then
${A}(\xi)=D\times {A}_P$. In the other cases, $\mathbf{Y}$ is
irreducible and ${A}_P={A}(\xi)$.
\end{coro}

\begin{coro}
 The group $A_P'$ is trivial,
except if  $\rc'$ is as in Proposition \ref{prop-1-d} (b) with
$\chi(\lambda_i)\neq\frac{\lambda_i+1}{2}$,
$\chi'(\lambda_i')=\chi(\lambda_i)-1$, and
$\chi(\lambda_{i+2})+\chi(\lambda_{i})=\lambda_{i+2}+1$ if $i<2s-1$,
$\chi(\lambda_{i})=m+1$ if $i=2s-1$. We have
${A}_P'=\{1,\epsilon_{i+1}'\epsilon_{i+2}'\}\subset {A}'(\xi')$ if
$i<2s-1$ and ${A}_P'=\{1,\epsilon_{2s}'\}\subset {A}'(\xi')$ if
$i=2s-1$.
\end{coro}

\subsection{}\label{ssec-uu}
Write $\xi=V_{2m+1}\oplus W$, where
$W=W_{l_1}(\lambda_1)\oplus\cdots\oplus W_{l_s}(\lambda_s)$,
$l_i=\chi(\lambda_i)$, $\lambda_i\geq\lambda_{i+1}$ (notation as in
\cite{X2}). Let $v_i$, $i=0,\ldots,m$ be as in \ref{ssec-or}. We
view $W$ as a $\tk[t]$ module by $\sum a_it^iw=\sum a_iT_\xi^iw$. It
follows from Lemma \ref{lem-p-2} that $\cP_\xi$ is identified with
$\mathbb{P}((\tk v_0\oplus\ker t)\cap\alpha^{-1}(0))$. Let
$\Sigma\in\bY$ and $\Sigma^\p=\{v'\in V|\beta(v',\Sigma)=0\}$. The
bilinear form $\beta_\xi$ induces a bilinear form
$\bar{\beta}_{\xi}$ on $\Sigma^\p/\Sigma$. Then $\bar{\beta}_\xi$
defines an element $\xi'\in\Lo(\Sigma^\p/\Sigma)^*\cong\Ll^*$. We
have that $\xi'\in\rc'$ and $\beta_{\xi'}=\bar{\beta}_{\xi}$. The
variety $\mathbf{Y}$ in various cases is described in the following.

(i) Assume $m\geq 1$ and $m-1\geq\lambda_1-l_1$.
\begin{eqnarray*}&&\xi'=V_{2m-1}\oplus
W_{l_1}(\lambda_1)\oplus\cdots\oplus W_{l_s}(\lambda_s),
\\
&&\mathbf{Y}=\{\tk v_0\} \text{ consists of one point}.
\end{eqnarray*}

(ii) Assume $\lambda_j-l_j-1\geq
\lambda_{j+1}-l_{j+1},\lambda_j\geq\lambda_{j+1}+1$. Then $m\geq 1$.
\begin{eqnarray*}
&&\xi'=V_{2m+1}\oplus W_{l_1}(\lambda_1)\oplus\cdots\oplus
W_{l_j}(\lambda_j-1)\oplus\cdots\oplus W_{l_s}(\lambda_s).
\\
&&\mathbf{Y}=\{\tk v|v=av_0+t^{\lambda_j-1}w,w\in
W,t^{\lambda_j}w=0,w\notin
tW,\alpha(t^{l_j-1}w)\neq a^2\delta_{m,\lambda_j-l_j}\}.\\
&&\dim \mathbf{Y}=2j.
\end{eqnarray*}

(iii) Assume $l_j-1\geq
l_{j+1},l_j\geq[\lambda_j/2]+1,\lambda_j\geq\lambda_{j+1}+1$.
\begin{eqnarray*}
&&\xi'=V_{2m+1}\oplus W_{l_1}(\lambda_1)\oplus\cdots\oplus
W_{l_j-1}(\lambda_j-1)\oplus\cdots\oplus W_{l_s}(\lambda_s),\\
&&{\small \mathbf{Y}\subset \mathbf{Y}':=\{\tk
v|v=av_0+t^{\lambda_j-1}w,w\in W,t^{\lambda_j}w=0,w\notin
tW,\alpha(t^{l_j-1}w)=a^2\delta_{m,\lambda_j-l_j}\},\text{ see \ref{ssec-y}}} \\
&&\dim \mathbf{Y}=2j-1.
\end{eqnarray*}

\subsection{}\label{ssec-y}
Case (i) is clear. We explain case (iii) in detail. Case (ii) is
similar.

Let $\Sigma=\tk v\in \mathbf{Y}$, where $v=av_0+t^{\lambda_j-1}w$.
Let $u_i\in V_{2m+1}$, $i=0,\ldots,m-1$ be as in \ref{ssec-or}.
Assume $a\neq 0$.
 There exists $w_0\in W$
such that $\beta(w_0,t^{\lambda_j-1}w)=a$ and (if $m\geq 1$) $\alpha(w_0)=0$. If $m=0$, let $\tilde{W}=\{w_1+\beta(w_1,w_0)v_0|w_1\in  W\}$; if $m\geq 1$, let
$\tilde{u}_0=w_0+u_0$ and we define $\tilde{V}_{2m+1}$, $\tilde{W}$
as in \ref{ssec-or}. Then $V=\tilde{V}_{2m+1}\oplus \tilde{W}$
($\chi_{\tilde{W}}(\lambda_i)=l_i$ with a careful choice of $w_0$) and $\Sigma\subset \tilde{W}$.
 Note that $v=t^{\lambda_j-1}(\tilde{w})$,
where $\tilde{w}=w+\sum_{i=0}^m\beta(w_0,t^iw)v_i\in\tilde{W}$ and
$\alpha(t^{l_j-1}\tilde{w})=\alpha(t^{l_j-1}w)+a^2\delta_{m,\lambda_j-l_j}$.

Now we can assume $V=V_{2m+1}\oplus W$ is a normal form of $\xi$,
with $\Sigma=\tk v\subset W$ and $v=t^{\lambda_j-1}w,w\in W$. Then
$\Sigma^\p/\Sigma=V_{2m+1}\oplus(\Sigma^\p\cap W)/\Sigma$. We apply
the results for orthogonal Lie algebras to $(\Sigma^\p\cap
W)/\Sigma$ (see \ref{ssec-yc}). Write $W'=(\Sigma^\p\cap W)/\Sigma$.
The set $\mathbf{Y}=\mathbf{Y}'$, except if
$l_{j-1}=l_j>\frac{\lambda_{j-1}+1}{2}$, $m>\lambda_{j-1}-l_{j-1}$
and for all $\lambda_a>\lambda_{j-1}$,
$\lambda_a-l_a>\lambda_j-l_j$, then $\mathbf{Y}$ consists of those
$v$ such that $\chi_{W'}(\lambda_{j-1})=l_{j-1}$.

\subsection{}
We prove Proposition \ref{prop-gps-d2}. In case (i), we have
$L=\{\tk v_0\}\subset V_{2m+1}$. For any $g\in Z_G(\xi)$, we have
that $gv_0=v_0$. Hence $H=Z_P(\xi)=Z_G(\xi)$, $K=Z_P(\xi)\cap
Z_G^0(\xi)=Z_G^0(\xi)$ and $A_P=A(\xi),A_P'=1$.

In cases (ii) and (iii), we can find a normal form $V=V_{2m+1}\oplus
W$ such that $\Sigma\subset W$ (see \ref{ssec-y}). Let $X(\Sigma)$
be the set of all such $W$. We first show that $Z_G(\xi)$ acts
transitively on $\mathbf{Y}$. Let
$\widetilde{\mathbf{Y}}=\{(\Sigma,W)|\Sigma\in \mathbf{Y},W\in
X(\Sigma)\}$. Then $Z_G(\xi)$ acts transitively on
$pr_2(\widetilde{\mathbf{Y}})$.  Set
$\widetilde{\mathbf{Y}}_W=pr_2^{-1}(W)=\{\Sigma\in \mathbf{Y}|W\in
X(\Sigma)\}$. It follows from the results in the orthogonal Lie
algebra case that $Z_W$ acts transitively on
$\widetilde{\mathbf{Y}}_W$ (see Proposition \ref{prop-gps}). Then
$Z_G(\xi)$ acts transitively on $\widetilde{\mathbf{Y}}$ and hence
acts transitively on $\mathbf{Y}=pr_1(\widetilde{\mathbf{Y}})$.

Fix $\Sigma\in \mathbf{Y}$ and $W\in X(\Sigma)$. Let $Z_W$ and
$Z_\Sigma$ be the stabilizer of $W$ and $\Sigma$ in $Z_G(\xi)$
respectively. The morphism $A_P\rightarrow A'(\xi')/A'_P$ is induced
by the natural morphism $Z_\Sigma/Z_\Sigma^0\rightarrow A'(\xi')$.
Since $X(\Sigma)$ is irreducible, $Z_{\Sigma,W}=Z_\Sigma\cap Z_W$
meets all the irreducible components of $Z_\Sigma$. Thus to study
the morphism $A_P\rightarrow A'(\xi')/A'_P$, it suffices to study
the natural morphism $Z_{\Sigma,W}/Z_{\Sigma,W}^0\rightarrow
A'(\xi')$.

The quadratic form $\alpha$ on $V$ restricts to nondegenerate
quadratic forms on $W$ and $W^\p$. Let $G(W)$, $G(W^\p)$ be the
groups preserving the respective quadratic forms and $\Lg(W)$,
$\Lg(W^\p)$ the Lie algebras. The bilinear form $\beta_\xi$ on $V$
restricts to bilinear forms on $W$ and $W^\p$. Let $\xi_W$ and
$\xi_{W^\p}$ be the corresponding elements in $\Lg(W)^*$ and
$\Lg(W^\p)^*$ respectively. Moreover, the bilinear form $\beta_\xi$
induces a bilinear form on $\tilde{W}=(\Sigma^\p\cap W)/\Sigma$. Let
$\xi_{\tilde{W}}$ be the corresponding element in $\Lg(\tilde{W})^*$
and
$A'(\xi_{\tilde{W}})=Z_{G(\tilde{W})}(\xi_{\tilde{W}})/Z_{G(\tilde{W})}^0(\xi_{\tilde{W}})$.

Let $Z=\{z\in Z_{G(W)}(\xi_W)|z\Sigma=\Sigma\}$. Since $Z_W\cong
Z_{G(W)}(\xi_W)\times Z_{G(W^\p)}(\xi_{W^\p})$, we have natural
isomorphisms $Z_{\Sigma,W}\cong Z\times Z_{G(W^\p)}({\xi_{W^\p}})$
and $Z_{\Sigma,W}/Z_{\Sigma,W}^0\cong Z/Z^0\times A(\xi_{W^\p})$.
Note that $A(\xi_{W^\p})=\{1\}$. On the other hand, we have a
morphism $A'(\xi_{\tilde{W}})\times A(\xi_{W^\p})\rightarrow
A'(\xi')$ which comes from the isomorphism $\Sigma^\p/\Sigma\cong
\tilde{W}\oplus W^\p$ and it is given by the system of generators.
It follows form the results for orthogonal Lie algebras that the
morphism $Z/Z^0\rightarrow A'(\xi_{\tilde{W}})$ is given by
generators (see Proposition \ref{prop-gps}). We then deduce easily
that $A_P'$ and the morphism $A_P\rightarrow A'(\xi')/A'_P$ are as
in Proposition \ref{prop-gps-d2}.

\section{Some combinatorics}\label{sec-com}
In this section we recall some combinatorics from \cite{Lu1,LS}. The
combinatorics goes back to \cite{Lu3}, where it is used to
parametrize unipotent representations of classical groups. We will
use the same kind of combinatorial objects to describe the Springer
correspondence for classical Lie algebras and their duals in
characteristic 2.
\subsection{}
Let $r,s,n\in\mathbb{N}=\{0,1,2,\ldots\}$, $d\in\mathbb{Z}$,
$e=[\frac{d}{2}]\in\mathbb{Z}$ ($[-]$ means the integer part). Let
$\tilde{X}_{n,d}^{r,s}$ be the set of all ordered pairs $(A,B)$ of
finite sequences of natural integers $A=(a_1,a_2,\ldots,a_{m+d})$
and $B=(b_1,b_2,\ldots,b_m)$ (for some $m$) satisfying the following
conditions:
\begin{eqnarray*}
&&a_{i+1}-a_i\geq r+s,\ i=1,\ldots,m+d-1\\
&&b_{i+1}-b_i\geq r+s,\ i=1,\ldots,m-1\\
&&b_1\geq s\\
&&\sum a_i+\sum b_i=n+r(m+e)(m+d-e-1)+s(m+e)(m+d-e).
\end{eqnarray*}
The set $\tilde{X}_{n,d}^{r,s}$ is equipped with a shift
$\sigma_{r,s}$. If $(A,B)$ is as above, then
$\sigma_{r,s}(A,B)=(A',B')$,
\begin{eqnarray*}A'=(0,a_1+r+s,\ldots,a_{m+d}+r+s),\
B'=(s,b_1+r+s,\ldots,b_{m}+r+s).\end{eqnarray*}

Let $X_{n,d}^{r,s}$ be the quotient of $\tilde{X}_{n,d}^{r,s}$ by
the equivalence relation generated by the shift and
$$X_n^{r,s}=\bigcup_{d\text{ odd}}X_{n,d}^{r,s}.$$
The equivalence class of $(A,B)$ is still denoted by $(A,B)$.

Assume $s=0$. Then there is an obvious bijection $X_{n,d}^{r,0}\to
X_{n,-d}^{r,0},(A,B)\mapsto(B,A)$. This induces an involution on
each of the following sets \beq X_{n,\text{even}}^r=\bigcup_{d\text{
even}}X_{n,d}^{r,0},\ X_{n,\text{odd}}^r=\bigcup_{d\text{
odd}}X_{n,d}^{r,0}.\eeq

Let $Y_{n,\text{even}}^r$ (resp. $Y_{n,\text{odd}}^r$) be the
quotient of $X_{n,\text{even}}^r$ (resp. $X_{n,\text{odd}}^r$) by
this involution. For $d\geq 0$, the image of $X_{n,\pm d}^{r,0}$ in
$Y_{n,\text{even}}^r$ or $Y_{n,\text{odd}}^r$ is denoted $Y_{n,d}^r$
and the image of $(A,B)$ is denoted $\{A,B\}$.

\subsection{}
When we consider simultaneously two elements $(A,B)\in
X_{n,d}^{r,s}$ and $(A',B')\in X_{n',d'}^{r',s'}$ with $d-d'$ even,
with $A=(a_1,\ldots,a_{m+d}),B=(b_1,\ldots,b_m)$ and
$A'=(a_1',\ldots,a'_{m'+d'}),B'=(b_1',\ldots,b'_{m'})$, we always
assume that we have chosen representatives such that $2m+d=2m'+d'$.
We use the same convention for $\{A,B\}\in Y_{n,d}^r$ and
$\{A',B'\}\in Y_{n',d'}^{r'}$ with $d,d'\geq0$ and $d-d'$ even.

There is an obvious addition \beq X_{n,d}^{r,s}\times
X_{n',d}^{r',s'}\rightarrow X_{n+n',d}^{r+r',s+s'},
(A,B)+(A',B')=(A'',B''),a_i''=a_i+a_i', b_i''=b_i+b_i'.\eeq The same
formula defines $Y_{n,d}^r\times Y_{n',d}^{r'}\rightarrow
Y_{n+n',d}^{r+r'}.$

Let $\Lambda_{0,1}^{r,s}\in X_{0,1}^{r,s}$ (resp.
$\Lambda_{0,0}^{r,s}\in X_{0,0}^{r,s}$) be the element represented
by $(A,B)=(0,\emptyset)$ (resp. $(A,B)=(\emptyset,\emptyset)$). If
$s=0$, let $\Lambda_{0,1}^r\in Y_{0,1}^r$ (resp. $\Lambda_{0,0}^r\in
Y_{0,0}^r$) be the image of $\Lambda_{0,1}^{r,0}$ (resp.
$\Lambda_{0,0}^r$). We have the following bijective maps:
\begin{eqnarray*}
&&X_{n,1}^{0,0}\rightarrow X_{n,1}^{r,s},\
\Lambda\mapsto\Lambda+\Lambda_{0,1}^{r,s},\quad
Y_{n,d}^{0}\rightarrow Y_{n,d}^{r},\
\Lambda\mapsto\Lambda+\Lambda_{0,d}^{r},d=0,1.
\end{eqnarray*}

Since $Y_{n,d}^0,d\geq1$, and $X_{n,d}^{0,0}$ are obviously in
bijection with the set of all pairs of partitions $(\mu,\nu)$ such
that $\sum\mu_i+\sum\nu_i=n$ and thus with $\mathbf{W}_n^\wedge$,
$Y_{n,0}^0$ is in bijection with the set of all unordered pairs of
partitions $\{\mu,\nu\}$ such that $\sum\mu_i+\sum\nu_i=n$ and thus
with ${\mathbf{W}_n^\wedge}'$, we get bijections \begin{eqnarray*}
&&\mathbf{W}_{n}^\wedge\xrightarrow{\sim} X_{n,1}^{r,s},\
\mathbf{W}_{n}^\wedge\xrightarrow{\sim} Y_{n,1}^{r},\
{{\mathbf{W}_{n}^{\wedge}}}'\xrightarrow{\sim} Y_{n,0}^{r}.
\end{eqnarray*}

\subsection{}
An element $(A,B)\in X_{n,d}^{r,s}$ is called distinguished if
$d=0$, $a_1\leq b_1\leq a_2\leq\cdots\leq a_m\leq b_m$ or if $d=1$,
$a_1\leq b_1\leq a_2\leq\cdots\leq a_m\leq b_m\leq a_{m+1}$. An
element $\{A,B\}\in Y_{n,d}^{r}$ $(d\geq 0)$ is called distinguished
if $(A,B)$ or $(B,A)$ is distinguished. Let
$D_n^{r,s},D_{n,\text{even}}^r,D_{n,\text{odd}}^r,
D_{n,d}^{r,s},D_{n,d}^r$ be the set of all distinguished elements in
$X_{n}^{r,s}$, $Y_{n,\text{even}}^r,Y_{n,\text{odd}}^r$,
$X_{n,d}^{r,s}$, $Y_{n,d}^r$ respectively.

Assume $r\geq 1$. For $(A,B)\in\tilde{X}_{n,d}^{r,s}$, we regard
$A,B$ as subsets of $\mathbb{N}$. Two elements $(A,B),(C,D)\in
X_n^{r,s}$ are said to be similar if $A\cup B=C\cup D$ and $A\cap
B=C\cap D$. We define similarity in $Y_{n,\text{even}}^r$ and
$Y_{n,\text{odd}}^r$ in the same way.

Let $S=(A\cup B)\backslash (A\cap B)$. A nonempty subset $I$ of $S$
is called an interval of $(A,B)$ or $\{A,B\}$ if it satisfies the
following conditions:
\begin{enumerate}
\item[(i)] if $i<j$ are consecutive elements of $I$, then $j-i<r+s$;
\item[(ii)] if $i\in I$, $j\in S$ and $|i-j|<r+s$, then $j\in I$.
\end{enumerate}
We call $I$ an initial interval if there exists $i\in I$
 such that $i<s$ and a proper interval otherwise.

Let $\mathcal{S}\subset X_n^{r,s}$ (resp. $Y_{n,\text{odd}}^{r}$ or
$Y_{n,\text{even}}^{r}$) be a similarity class and $(A,B)$ (resp.
$\{A,B\}$)$\in\mathcal{S}$. Let $E$ be the set of all proper
intervals of $(A,B)$ (resp. $\{A,B\}$). The set $\mathcal{A}(E)$ of
all subsets of $E$ is a vector space over $\tF_2$. If
$\mathcal{S}\subset X_n^{r,s}$, it acts simply transitively on
$\mathcal{S}$ as follows. The image of $(A,B)$ under $F\subset E$ is
the pair $(C,D)$ such that $$A\cap I=D\cap I,B\cap I=C\cap I\text{
if and only if }I\in F.$$ If $\mathcal{S}\subset
Y_{n,\text{odd}}^{r}$ (or $Y_{n,\text{even}}^{r}$), as $E$
transforms $(A,B)$ to $(B,A)$, the same formula defines a simply
transitive action of $\mathcal{A}(E)/\{\emptyset,E\}$ on
$\mathcal{S}$.

For $\Lambda\in X_n^{r,s}$ (resp. $Y_{n,\text{odd}}^{r}$ or
$Y_{n,\text{even}}^{r}$), let $V_{\Lambda}^{r,s}$ (resp.
$V_\Lambda^r$) denote the vector space $\mathcal{A}(E)$ (resp.
$\mathcal{A}(E)/\{\emptyset,E\}$), where $E$ is the set of all
proper intervals of $\Lambda$. For $F\in V_{\Lambda}^{r,s}$ (resp.
$V_\Lambda^r$), let $\Lambda_F$ be the image of $\Lambda$ under the
action of $F$.

\subsection{Examples}
(1) $X_{n,1}^{1,0}$ and $Y_{n,0}^1$ are used in \cite{Lu3} to
describe $\mathbf{W}_n^\wedge$ and ${\mathbf{W}_n^{\wedge}}'$
respectively.

(2) Assume $\text{char}(\tk)\neq 2$. $X_n^{1,1},Y_{n,\text{even}}^2$
and $Y_{n,\text{odd}}^2$ are used in \cite{Lu1} to describe the
generalized Springer correspondence for $Sp({2n})$, $SO(2n)$ and
$SO(2n+1)$ respectively.

(3) Assume $\text{char}(\tk)=2$. $X_n^{2,2}$ and $Y_{n,\text{even}}^4$ are used in \cite{LS} to
describe the generalized Springer correspondence for unipotent
classes of $Sp({2n})$ (or $SO(2n+1)$) and $SO(2n)$ respectively.

(4) Assume $\text{char}(\tk)=2$. $Y_{n-1,\text{odd}}^{4}$, $X_n^{3,1}$ and
$X_{n,\text{even}}^{3,1}=\cup_{d\text{ even}}X_{n,d}^{3,1}$ are used
in \cite{Lu4} to describe the generalized Springer correspondence
for disconnected groups $O(2n), G_{2n+1}$ with $G^0$ type $A_{2n}$,
and $G_{2n}$ with $G^0$ type $A_{2n-1}$ respectively.

(5) Assume $\text{char}(\tk)=2$. We will use  $X_n^{2,n+1}, X_{n}^{n+1,n+1}$ and
$Y_{n,\text{even}}^{n+1}$ to describe the Springer correspondence
for $\mathfrak{o}(2n+1)$, $\mathfrak{sp}({2n})$ and $\Lo(2n)$ (or
$\Lo(2n)^*$). The set $D_n^{2,n+1}$ (resp. $D_{n}^{n+1,n+1}$,
$D_{n,\text{even}}^{n+1}$) is in bijection with the set of $O(2n+1)$
(resp. $Sp(2n)$, $O(2n)$)-nilpotent orbits in $\mathfrak{o}(2n+1)$
(resp. $\mathfrak{sp}({2n})$, $\Lo(2n)$).

(6) Assume $\text{char}(\tk)=2$. We will use  $Y_{n,\text{odd}}^{n+1}$ and $X_{n}^{1,n+1}$ to
describe the Springer correspondence for $\mathfrak{o}(2n+1)^*$ and
$\mathfrak{sp}({2n})^*$. The set $D_{n,\text{odd}}^{n+1}$ (resp.
$D_{n}^{1,n+1}$) is in bijection with the set of $O(2n+1)$ (resp.
$Sp(2n)$)-nilpotent orbits in $\mathfrak{o}(2n+1)^*$ (resp.
$\mathfrak{sp}({2n})^*$).

\section{Springer correspondence for symplectic Lie algebras}\label{sec-sy}
Assume that $G=Sp(2n)$ in this section.
\subsection{} Let $x\in\Lg$ be nilpotent. The orbit $\rc$ of
$x$ is characterized by the following data (\cite{Hes}):

(d1) The sizes of the Jordan blocks of $x$ give rise to a partition
of $2n$. We write it as
$\lambda_1\leq\lambda_2\leq\cdots\leq\lambda_{2s+1}$, where
$\lambda_1=0$.

(d2) For each $\lambda_i$, there is an integer $\chi(\lambda_i)$
satisfy $0\leq\chi(\lambda_i)\leq\frac{\lambda_i}{2}$. Moreover,
$\chi(\lambda_i)\geq\chi(\lambda_{i-1}),\lambda_i-\chi(\lambda_i)\geq\lambda_{i-1}-\chi(\lambda_{i-1})$,
$i=2,\ldots,2s+1$.

We can partition the set $\{1,2,\ldots,2s+1\}$ in a unique way into
blocks of length 1 or 2 such that the following holds:

(b1) If $\chi(\lambda_i)=\lambda_i/2$, then $\{i\}$ is one block;

(b2) All other blocks consist of two consecutive integers.
\newline Note that if $\{i,i+1\}$ is a block, then
$\lambda_i=\lambda_{i+1}$ and $\chi(\lambda_i)=\chi(\lambda_{i+1})$.

We attach to the orbit $\rc$ the sequence $c_1,\ldots,c_{2s+1}$
defined as follows:

(1) If $\{i\}$ is a block, then $c_i=\lambda_i/2+(n+1)(i-1)$;

(2) If $\{i,i+1\}$ is a block, then
$c_i=\lambda_i-\chi(\lambda_i)+(n+1)(i-1)$,
$c_{i+1}=\chi(\lambda_{i+1})+(n+1)i$.

Taking $a_i=c_{2i-1}$, $i=1,\ldots,s+1$, $b_i=c_{2i}$,
$i=1,\ldots,s$, we get a well defined element $(A,B)\in
X_{n,1}^{n+1,n+1}$. We denote it $\rho_G(x)$, $\rho(x)$ or
$\rho(\rc)$.

\begin{lem}
$\mathrm{(i)}$ $\rc\mapsto\rho(\rc)$ defines a bijection from the
set of all nilpotent $Sp(2n)$-orbits in $\mathfrak{sp}(2n)$ to
$D_{n}^{n+1,n+1}$.

$\mathrm{(ii)}$ $A_G(x)^\wedge$ is isomorphic to
$V_{\rho(x)}^{n+1,n+1}$.
\end{lem}
\begin{proof}
(i) It is easily checked from the definition  that $\rho(\rc)\in
D_{n}^{n+1,n+1}$ and the map $\rc\mapsto\rho(\rc)$ is injective.
Note that $X_{n,1}^{n+1,n+1}=D_{n}^{n+1,n+1}$ is in bijection with
$\mathbf{W}_n^\wedge$ and the number of nilpotent orbits is equal to
$|\mathbf{W}_n^\wedge|$ by Spaltenstein \cite{Spal}. Hence the
bijectivity of the map follows. In fact, given $(A,B)\in
D_{n}^{n+1,n+1}$, the corresponding nilpotent orbit can be obtained
as follows. Let $c_1\leq c_2\leq\cdots\leq c_{2s+1}$ be the sequence
$a_1\leq b_1\leq\cdots\leq a_{s+1}$. If $c_{i+1}<c_{i}+(n+1)$, then
$\{i,i+1\}$ is a block. We can recover $\lambda_i=\lambda_{i+1}$ and
$\chi(\lambda_{i})=\chi(\lambda_{i+1})$ from (2) of the definition.
All blocks of length 2 are obtained in this way. For the other
blocks, we can recover $\lambda_i$ and thus
$\chi(\lambda_{i})=\lambda_{i}/2$ from (1) of the definition.

(ii) One easily checks that $(A,B)$ has no proper intervals. It
follows that $V_{\rho(x)}^{n+1,n+1}=\{0\}$. On the other hand,
$A(x)=1$ since $Z_G(x)$ is connected by Spaltenstein \cite{Spal}.
\end{proof}

\subsection{}\label{thm-sc-sp}Consider a pair $(x,\phi)\in\mathfrak{A}_\Lg$, then
$\phi=1$.
\begin{thm}
The Springer correspondence $\gamma:\mathfrak{A}_\Lg\rightarrow
\mathbf{W}_n^\wedge\cong X_{n}^{n+1,n+1}$ is given by
$$(x,1)\mapsto\rho(x).$$
\end{thm}
\begin{rmk}
Theorem rewrites the description of Springer correspondence given by
Spaltenstein \cite{Spal} using pairs of partitions. Note that he
works under the assumption that the theory of Springer
representations is valid for $\Lg$ in characteristic 2.
\end{rmk}

\subsection{}\label{lemma-1}Let $\rc_{reg}$ be the nilpotent $G$-orbit in $\Lg$
which is open dense in the nilpotent variety $\cN$ of $\Lg$. Let
$\rc_{0}$ be the $0$ orbit.

\begin{lem}
The pair $(\rc_{reg},\bar{\bQ}_l)$ corresponds to the unit
representation and the pair $(\rc_{trivial},\bar{\bQ}_l)$
corresponds to the sign representation.
\end{lem}
\begin{proof}
One can show that the Weyl group action on $H^i(\cB)$ defined in
\cite{X1} coincides with the classical action. Assume $\chi\in
\mathbf{W}_n^\wedge$ correspond to the pair
$(\rc,\cF)\in\mathfrak{A}_\Lg$. We write $\chi=\chi_{(\rc,\cF)}$.
Recall that we have the following decomposition
$$\varphi_!\bar{\bQ}_{lX}|_{\cN}[\dim
G-n]=\bigoplus_{(\rc,\cF)\in\mathfrak{A}_\Lg}\chi_{(\rc,\cF)}\otimes
IC(\bar{\rc},\cF)[\dim c] .$$ Thus for $x\in \cN$,
$$H^{2i}(\cB_x,\bar{\bQ}_l)=\bigoplus_{(\rc,\cF)}\chi_{(\rc,\cF)}\otimes
(\cH^{2i+\dim\rc-\dim G+n}IC(\bar{\rc},\cF))_x.$$ Taking $i=(\dim
G-n)/2$ and $x=0$, we get $H^{\dim
G-n}(\cB,\bar{\bQ}_l)=\chi_{(\rc_0,\bar{\bQ}_l)}$ since
$A(\rc_0)=1$. It follows that $\chi_{(\rc_0,\bar{\bQ}_l)}$ is the
sign representation. Taking $i=0$ and $x=0$, we get
$H^{0}(\cB,\bar{\bQ}_l)=\chi_{(\rc_{reg},\bar{\bQ}_l)}\otimes
(\cH^{0}IC(\overline{\rc_{reg}},\bar{\bQ}_l))_0$ since
$A(\rc_{reg})=1$. It follows that $\chi_{(\rc_{reg},\bar{\bQ}_l)}$
is the unit representation.
\end{proof}
\begin{proof}[Proof of Theorem \ref{thm-sc-sp}]
By the discussion in   \ref{ss1-3}, it is enough to show that the
map $\gamma$ is compatible with the restriction formula
$(\mathbf{R})$. When $n=1$, by Lemma \ref{lemma-1}, the pair
$(\rc_{reg},1)$ corresponds to the unit representation and the pair
$(\rc_0,1)$ corresponds to the sign representation. When $n=2$,
there are two representations of $\mathbf{W}_2$ restricting to unit
representation and two representations of $\mathbf{W}_2$ restricting
to sign representation. But again we know the pair $(\rc_{reg},1)$
corresponds to the unit representation and the pair $(\rc_0,1)$
corresponds to the sign representation. When $n\geq 3$, we show that
the map $\gamma$ is compatible with the restriction formula. Let
$x\in\Lg$ and $x'\in\Ll$ be nilpotent elements. Note that we have
$A_G(x)=A_L(x')=1$. Hence it is enough to show
that\begin{equation}\label{eqn-3} \la 1,\varepsilon_{x,x'} \ra=\la
\text{Res}^{\mathbf{W}_n}_{\mathbf{W}_{n-1}}\rho_G(x),\rho_L(x')\ra_{\mathbf{W}_{n-1}}.
\end{equation}

Note that $X_{n,d}^{n+1,n+1}=\emptyset$ if $d\neq 1$ is odd. Thus
$X_{n}^{n+1,n+1}=X_{n,1}^{n+1,n+1}$. Let $(A,B)\in X_{n}^{n+1,n+1}$
correspond to $\chi\in \mathbf{W}_n^\wedge$. The pairs $(A',B')\in
X_{n-1}^{n,n}$ which correspond to the components of the restriction
of $\chi$ to $\mathbf{W}_{n-1}$ are those which can be deduced from
$(A,B)$ by decreasing one of the entries $c_i$ by $i$ and decreasing
all other entries $c_j$ by $j-1$. This can be done if and only if
$i\geq 3$ and $c_i-c_{i-2}\geq 2n+3$, $i=2$, $c_i\geq n+2$ or $i=1$,
$c_1\geq 1$. We write $(A,B)\to(A',B')$ if they are related in this
way.

Now (\ref{eqn-3}) follows since $S_{x,x'}\neq\emptyset$ if and only
if $x,x'$ are as in Proposition \ref{prop-1-sp} (see below) if and
only if $\rho_G(x)\to\rho_L(x')$.
\end{proof}

\subsection{}\label{prop-1-sp}
Consider a nilpotent class $\rc'\in f_x(\cP_x)$ corresponding to
$(\lambda_{2s+1}')_{\chi'(\lambda_{2s+1}')}\cdots(\lambda_1')_{\chi'(\lambda_{1}')}:=(\lambda',\chi')$.
Suppose $Y=f_x^{-1}(\rc')$ and $X=\varrho_x^{-1}(Y)$. (Notations are
as in \ref{ss1-2}.)

\begin{prop}[\cite{Spal}]
The group $Z_G(x)$ acts transitively on $Y$. We have $\dim X=\dim
\cB_x$ if and only if $(\lambda',\chi')$ satisfies a) or b):

\noindent a) Assume $\lambda_i-\lambda_{i-1}\geq 2$,
$\chi(\lambda_i)=\lambda_i/2$ and
$\chi(\lambda_j)\geq\lambda_j-\lambda_i/2+1$ for each $j<i$.
$\lambda_j'=\lambda_j$, $j\neq i$, $\lambda_i'=\lambda_i-2$,
$\chi'(\lambda_j')=\chi(\lambda_j)$ for each $j\neq i$ and
$\chi'(\lambda_i')=\lambda'_i/2$. In this case $\dim Y=2s-i+1$.

\noindent b) Assume $\lambda_{i+1}=\lambda_i>\lambda_{i-1}$.
$\lambda_j'=\lambda_j,j\neq i,i+1$,
$\lambda_{i+1}'=\lambda_i'=\lambda_i-1$,
$\chi'(\lambda_j')=\chi(\lambda_j)$ for each $j\neq i,i+1$ and
$\chi'(\lambda_i')=\chi'(\lambda_{i+1}')\in\{\chi(\lambda_i),\chi(\lambda_i)-1\}$
satisfy $0\leq\chi'(\lambda_i')\leq\lambda_i'/2$,
$\chi(\lambda_{i-1})\leq\chi'(\lambda_i')\leq\chi(\lambda_{i-1})+\lambda_i-\lambda_{i-1}-1$.
We have $\dim Y=2s-i+1$ if $\chi'(\lambda_i')=\chi(\lambda_i)$ and
$\dim Y=2s-i$ if $\chi'(\lambda_i')=\chi(\lambda_i)-1$.
\end{prop}

\section{Springer correspondence for orthogonal Lie
algebras}\label{sec-scor}

\subsection{}\label{lem-com-oor} In this subsection we assume that $G=O(2n+1)$.

Let
$x=(\lambda_{2s+1})_{\chi(\lambda_{2s+1})}\cdots(\lambda_1)_{\chi(\lambda_1)}\in\Lg$
be a nilpotent element (see \ref{comp-or}). Assume $\lambda_1=0$.
There exists a unique $3\leq m_0\leq 2s+1$ such that ${m_0}$ is odd
and $\lambda_{m_0}>\lambda_{{m_0}-1}$. We have
$\chi(\lambda_j)=\lambda_j$ if $j\leq{m_0}$;
$\lambda_{2j}=\lambda_{2j+1}$, $j\neq\frac{{m_0}-1}{2}$ and
$\lambda_{m_0}=\lambda_{{m_0}-1}+1$.

We attach to the orbit $\rc$ of $x$ the sequence
$c_1,\ldots,c_{2s+1}$ defined as follows:

(1)
$c_{2j}=\left\{\begin{array}{ll}\lambda_{2j}-\chi(\lambda_{2j})+n+1+(j-1)(n+3)
& \text{if } 2j<{m_0}\\
\lambda_{2j}-\chi(\lambda_{2j})+1+n+1+(j-1)(n+3) & \text{if } 2j\geq
{m_0}\end{array}\right.$

(2)
$c_{2j-1}=\left\{\begin{array}{ll}\chi(\lambda_{2j-1})+(j-1)(n+3) &
\text{if } 2j-1<{m_0}\\ \chi(\lambda_{2j-1})-1+(j-1)(n+3) & \text{if
} 2j-1\geq {m_0}\end{array}\right.$.

Taking $a_i=c_{2i-1}$, $i=1,\ldots,s+1$, $b_i=c_{2i}$,
$i=1,\ldots,s$, we get a well-defined element $(A,B)\in
X_{n,1}^{2,n+1}$. We denote it $\rho_G(x)$, $\rho(x)$ or
$\rho(\rc)$.

\begin{lem}
$\mathrm{(i)}$ $\rc\mapsto\rho(\rc)$ defines a bijection from the
set of all nilpotent $O(2n+1)$-orbits in $\mathfrak{o}(2n+1)$ to
$D_{n}^{2,n+1}$.

$\mathrm{(ii)}$ $A_G(x)^\wedge$ is isomorphic to
$V_{\rho(x)}^{2,n+1}$.
\end{lem}
\begin{proof}
(i) It is easily checked from the definition  that $\rho(\rc)\in
D_{n}^{2,n+1}$ and the map $\rc\mapsto\rho(\rc)$ is injective. Note
that $D_{n}^{2,n+1}$ is in bijection with the set $\Delta$
consisting of all pairs of partitions $(\mu,\nu)$ such that
$\sum\mu_i+\sum\nu_i=n,\nu_i\leq\mu_i+2$. Since the number of
nilpotent orbits is equal to $|\Delta|$ by Spaltenstein \cite{Spal},
the bijectivity of the map follows. In fact, given $(A,B)\in
D_{n}^{2,n+1}$, the corresponding nilpotent orbit can be obtained as
follows. Let $c_1\leq c_2\leq\cdots\leq c_{2s+1}$ be the sequence
$a_1\leq b_1\leq\cdots\leq a_{s+1}$. There exists a unique odd
integer $m_0$ such that $c_{2j}>(n+1)+(j-1)(n+3)$ if and only if
$2j>m_0$. If $j<\frac{m_0-1}{2}$, then
$\lambda_{2j}=\lambda_{2j+1}=\chi(\lambda_{2j})=\chi(\lambda_{2j+1})=c_{2j+1}-j(n+3)$.
If $j>\frac{m_0-1}{2}$, then
$\lambda_{2j}=\lambda_{2j+1}=c_{2j}+c_{2j+1}-(2j-1)(n+3)-(n+1)$ and
$\chi(\lambda_{2j})=\chi(\lambda_{2j+1})=c_{2j+1}-j(n+3)+1$. If
$j=\frac{m_0-1}{2}$, then
$\lambda_{2j}=\chi(\lambda_{2j})=\lambda_{2j+1}-1=\chi(\lambda_{2j+1})-1=c_{2j+1}-j(n+3)+1$.

(ii) The component group $A_G(x)$ is described in \ref{comp-or}. Let
$(A,B)=\rho(x)$ and $c_1,\ldots,c_{2s+1}$ be as above. Let $S=(A\cup
B)\backslash (A\cap B)$. Note that $c_1=0,c_2,\ldots,c_{m_0}$ all
lie in $ S$ and they belong to the same interval, which is the
initial interval. For $i>m_0$, $\chi(\lambda_i)\neq\lambda_i/2$ if
and only if $c_i\in S$. The relations (r2) and (r3) of \ref{comp-or}
say that if $c_i,c_j$ belong to the same interval of $(A,B)$, then
$\epsilon_i,\epsilon_j$ have the same images in $A(x)$. Thus we get
an element $\sigma_I$ of $A(x)$ for each interval $I$ of $(A,B)$ and
$\sigma_I^2=1$. Moreover (r3) means that $\sigma_I=1$ if $I$ is the
initial interval.

The isomorphism $V_{\rho(x)}^{2,n+1}\rightarrow A_G(x)^\wedge$ is
given as follows. Let $F\in V_{\rho(x)}^{2,n+1}$. We associate to
$F$ the character of $A_G(x)$ which takes value $-1$ on $\sigma_I$
if and only if $I\in F$.
\end{proof}

\subsection{}\label{prop-sc3}Let $(x,\phi)\in\mathfrak{A}_\Lg$. We have defined
$\rho(x)$. Let $\rho$ denote also the map $A_G(x)^\wedge\rightarrow
V_{\rho(x)}^{2,n+1}$.
\begin{thm}
The Springer correspondence $\gamma:\mathfrak{A}_\Lg\rightarrow
\mathbf{W}_n^\wedge\cong X_n^{2,n+1}$ is given by
$$(x,\phi)\mapsto\rho(x)_{\rho(\phi)}.$$
\end{thm}
\begin{proof} As in the proof of
Theorem \ref{thm-sc-sp}, it is enough to prove the map $\gamma$ is
compatible with the restriction formula $(\mathbf{R})$. Note that
$X_{n,d}^{2,n+1}=\emptyset$ if $d\neq 1$ is odd. Thus
$X_{n}^{2,n+1}=X_{n,1}^{2,n+1}$. Let $(A,B)\in X_{n}^{2,n+1}$
correspond to $\chi\in \mathbf{W}_n^\wedge$. The pairs $(A',B')\in
X_{n-1}^{2,n}$ which correspond to the components of the restriction
of $\chi$ to $\mathbf{W}_{n-1}$ are those which can be deduced from
$(A,B)$ by decreasing one of the entries $a_{i}$ by $i$ (or $b_{i}$
by $i+1$) and decreasing all other entries $a_j$ by $j-1$, $b_j$ by
$j$. We can decrease $a_{i}$ by $i$ (resp. $b_{i}$ by $i+1$) if and
only if $i\geq 2$, $a_i-a_{i-1}\geq n+4$ or $i=1$, $a_i\geq 1$
(resp. $i\geq2$, $b_i-b_{i-1}\geq n+4$ or $i=1$, $b_i\geq n+2$). We
write $(A,B)\to(A',B')$ if they are related in this way. Suppose
that $(A,B)\to(A',B')$. One can easily check that if $(A,B)$ and
$(A',B')$ are similar to $\Lambda\in D_n^{2,n+1}$ and $\Lambda'\in
D_{n-1}^{2,n+1}$ respectively, then $\Lambda\to\Lambda'$.

Let $x\in\Lg$ nilpotent and $x'\in\Ll$ nilpotent. Then
$S_{x,x'}\neq\emptyset$ if and only if $x,x'$ are as in Proposition
\ref{prop-1} if and only if
$\Lambda=\rho_G(x)\to\Lambda'=\rho_L(x')$. To verify the map is
compatible with the restriction formula, it is enough to show that
the set \begin{equation}\label{eqn-4}\{(F,F')\in
V_{\rho_G(x)}^{2,n+1}\times
V_{\rho_L(x')}^{2,n}|\Lambda_F\to\Lambda'_{F'}\}\end{equation} is
the image of the set \begin{equation}\label{eqn-5} \{(\phi,\phi')\in
A_G(x)^\wedge\times A_L(x')^\wedge|\la
\phi\otimes\phi',\varepsilon_{x,x'}\ra\neq 0\}\end{equation} under
the map $\rho$.

Let $c_1\leq\cdots\leq c_{2s+1}$ and $c_1'\leq\cdots\leq c_{2s+1}'$
correspond to $\Lambda$ and $\Lambda'$ respectively. Then $A_G(x)$
is generated by $\{\epsilon_i|c_i\neq c_j,\forall j\neq i\}$,
$A_L(x')$ is generated by $\{\epsilon_i'|c_i'\neq c_j',\forall j\neq
i\}$ and $A_P$ is generated by $\{\epsilon_i|c_i\neq c_j,c_i'\neq
c_j',\forall j\neq i\}$. There are various cases to consider. We
describe one of the cases in the following and the other cases are
similar.

Assume $c_{2k+1}>c_{2k}+1$, $c_{2k+2}=c_{2k+1}+n+2$ and
$c_{2k+1}'=c_{2k+1}-(k+1)$, $c_{2i+1}'=c_{2i+1}-i$,$i\neq k$,
$c_{2i}'=c_{2i}-i$, $i=1,\ldots,s$. Let $I$ (resp. $I'$) be the
interval of $\Lambda$ (resp. $\Lambda'$) containing $c_{2k+1}$
(resp. $c_{2k+1}'$) and $J'$ the interval of $\Lambda'$ containing
$c_{2k+2}'$. Note that $c_{2j+1}-c_{2j}<n+2$, except if
$x=(n+1)_{n+1}n_n$. In the latter case $A(x)=1$. Moreover
$c_{2j+2}-c_{2j+1}\geq 2$. Hence all other intervals of $\Lambda$
and $\Lambda'$ can be identified naturally. There are two
possibilities:

(i) $I$ is a proper interval of $\Lambda$. Then
$\Lambda_F\to\Lambda_{F'}$ if and only if
\begin{itemize}
\item[(a)]$F\backslash\{I\}=F'\backslash\{I',J'\}$;
\item[(b)]$F\cap\{I\}=F'\cap\{I',J'\}=\emptyset$ or $\{I\}\subset F,\{I',J'\}\subset
F'$.
\end{itemize}
On the other hand, $A_G(x)$ (resp. $A_L(x')$) is an $\tF_2$ vector
space with one basis element $\sigma_K$ (resp. $\sigma_{K}'$) for
each proper interval $K$ of $\Lambda$ (resp. $\Lambda'$) and
$S_{x,x'}$ is the quotient of $A_G(x)\times A_L(x')$ by the subgroup
$H_{x,x'}$ generated by elements of the form
$\sigma_I\sigma'_{I'},\sigma_I\sigma'_{J'},\sigma_K\sigma'_K$ with
$K$ a proper interval of both $\Lambda$  and $\Lambda'$. Now the
compatibility between (\ref{eqn-4}) and (\ref{eqn-5}) is clear.

(ii) $I$ is an initial interval of $\Lambda$. Then
$\Lambda_F\to\Lambda_{F'}$ if and only if $F=F'$. On the other hand,
$A_G(x)$, $A_L(x')$ and $S_{x,x'}$ are obtained by setting
$\sigma_I=\sigma'_{I'}=1$ in (i). Again the compatibility between
(\ref{eqn-4}) and (\ref{eqn-5}) is clear.
\end{proof}

\subsection{}\label{ss-scor2}In this subsection we assume $G=SO(2n)$,
$\tilde{G}=O(2n)$ and $\Lg=\mathfrak{o}(2n)$. We describe
$\tilde{\gamma}:\tilde{\mathfrak{A}}_\Lg\rightarrow
{\mathbf{W}_n^\wedge}'$ instead of
$\gamma:\mathfrak{A}_\Lg\rightarrow (\mathbf{W}_n')^\wedge$ (see
\ref{ssec-weylgp}).

Let
$x=(\lambda_{2s})_{\chi(\lambda_{2s})}\cdots(\lambda_{1})_{\chi(\lambda_{1})}\in\Lg$
be a nilpotent element (see \ref{comp-or}). Note that we have
$\lambda_{2i-1}=\lambda_{2i}$. We attach to the orbit $\rc$ of $x$
the sequence $c_1,\ldots,c_{2s}$ defined as follows:

(1) $c_{2j}=\chi(\lambda_{2j})+(j-1)(n+1)$,

(2) $c_{2j-1}=\lambda_{2j-1}-\chi(\lambda_{2j-1})+(j-1)(n+1)$.

Taking $a_i=c_{2i-1}$, $b_i=c_{2i}$, $i=1,\ldots,s$, we get a well
defined element $\{A,B\}\in Y_{n,0}^{n+1}$. We denote it
$\rho_G(x)$, $\rho(x)$ or $\rho(\rc)$.

\begin{lem}
$\mathrm{(i)}$ $\rc\mapsto\rho(\rc)$ defines a bijection from the
set of all nilpotent $O(2n)$-orbits in $\mathfrak{o}(2n)$ to
$D_{n,\text{even}}^{n+1}$.

$\mathrm{(ii)}$ $A_G(x)^\wedge$ is isomorphic to
$V_{\rho(x)}^{n+1}$.
\end{lem}
\begin{proof}
(i) It is easily checked from the definition  that $\rho(\rc)\in
D_{n,\text{even}}^{n+1}$ and the map $\rc\mapsto\rho(\rc)$ is
injective. Note that $D_{n,\text{even}}^{n+1}$ is in bijection with
the set $\Delta$ consisting of all pairs of partitions $(\mu,\nu)$
such that $\sum\mu_i+\sum\nu_i=n,\nu_i\leq\mu_i$. Since the number
of nilpotent $O(2n)$-orbits in $\mathfrak{o}(2n)$ is equal to
$|\Delta|$ by Spaltenstein \cite{Spal}, the bijectivity of the map
follows. In fact, given $\{A,B\}\in D_{n,\text{even}}^{n+1}$ with
preimage $(A,B)\in D_{n,0}^{n+1,0}$, the corresponding nilpotent
orbit can be obtained as follows. Let $c_1\leq c_2\leq\cdots\leq
c_{2s}$ be the sequence $a_1\leq b_1\leq\cdots\leq a_s\leq b_{s}$.
We have $\lambda_{2j}=\lambda_{2j-1}=c_{2j}+c_{2j-1}-(2j-2)(n+1)$
and $\chi(\lambda_{2j})=\chi(\lambda_{2j-1})=c_{2j}-(j-1)(n+1)$.

(ii) The component group $A_G(x)$ is described in \ref{comp-or}.
Note that in this case, the condition (r3) is void. By similar
argument as in the proof of Lemma \ref{lem-com-oor} (ii), one shows
that $A_{\tilde{G}}(x)$ is a vector space over $\tF_2$ with basis
$(\sigma_I)_{I\in E}$, where $E$ is the set of all intervals of
$\rho(x)$.  Since $A_{G}(x)$ consists of the elements in
$A_{\tilde{G}}(x)$ which can be written as a product of even number
of generators, from the natural identification
$A_{\tilde{G}}(x)^\wedge=\mathcal{A}(E)$, we get the isomorphism
$A_G(x)^\wedge\cong\mathcal{A}(E)/\{\emptyset,
E\}=V_{\rho(x)}^{n+1}$.
\end{proof}

\subsection{}\label{prop-sc4}Let $(x,\phi)\in\tilde{\mathfrak{A}}_\Lg$. We have defined
$\rho(x)$. Let $\rho$ denote also the map
$A_{G}(x)^\wedge\rightarrow V_{\rho(x)}^{n+1}$.
\begin{thm}
The Springer correspondence $
\tilde{\gamma}:\tilde{\mathfrak{A}}_\Lg\rightarrow
{\mathbf{W}_n^\wedge}'\cong Y_{n,\text{even}}^{n+1}$ is given by
$$(x,\phi)\mapsto\rho(x)_{\rho(\phi)}.$$
\end{thm}
\begin{proof} Again it is enough to prove the map $\tilde{\gamma}$ is compatible
with the restriction formula $(\mathbf{R})$. Note that
$Y_{n,d}^{n+1}=\emptyset$ if $d>0$ is even. Thus
$Y_{n,\text{even}}^{n+1}=Y_{n,0}^{n+1}$. Let $\{A,B\}\in
Y_{n,\text{even}}^{n+1}$ correspond to $\chi\in
{\mathbf{W}_n^{\wedge}}'$. The pairs $\{A',B'\}\in
Y_{n-1,\text{even}}^{n}$ which correspond to the components of the
restriction of $\chi$ to $\mathbf{W}_{n-1}'$ are those which can be
deduced from $\{A,B\}$ by decreasing one of the entries $a_{i}$ by
$i$ (or $b_{i}$ by $i$) and decreasing all other entries $a_j$ by
$j-1$, $b_j$ by $j-1$. We can decrease $a_{i}$ by $i$ (resp. $b_{i}$
by $i$) if and only if $i\geq 2$, $a_i-a_{i-1}\geq n+2$ or $i=1$,
$a_i\geq 1$ (resp. $i\geq 2$, $b_i-b_{i-1}\geq n+2$ or $i=1$,
$b_i\geq 1$). We write $\{A,B\}\to\{A',B'\}$ if they are related in
this way. Suppose that $\{A,B\}\to\{A',B'\}$. One can easily check
that if $\{A,B\}$ and $\{A',B'\}$ are similar to $\Lambda\in
D_{n,\text{even}}^{n+1}$ and $\Lambda'\in D_{n-1,\text{even}}^{n}$
respectively, then $\Lambda\to\Lambda'$.

Let $x\in\Lg$ nilpotent and $x'\in\Ll$ nilpotent. Then
$\tilde{S}_{x,x'}\neq\emptyset$ ($\Leftrightarrow
S_{x,x'}\neq\emptyset$) if and only if $x,x'$ are as in Proposition
\ref{prop-1} if and only if
$\Lambda=\rho_G(x)\to\Lambda'=\rho_L(x')$. Let $c_1\leq\cdots\leq
c_{2s}$ and $c_1'\leq\cdots\leq c_{2s}'$ correspond to $\Lambda$ and
$\Lambda'$ respectively. Then $A_{\tilde{G}}(x)$ is generated by
$\{\epsilon_i|c_i\neq c_j,\forall j\neq i\}$, $A_{\tilde{L}}(x')$ is
generated by $\{\epsilon_i'|c_i'\neq c_j',\forall j\neq i\}$ and
$\tilde{A}_P$ is generated by $\{\epsilon_i|c_i\neq c_j,c_i'\neq
c_j',\forall j\neq i\}$. The discussion in \ref{ssec-sxx} allows us
to compute $\varepsilon_{x,x'}$ and the set $\{(\phi,\phi')\in
A_G(x)^\wedge\times A_L(x')^\wedge|\la
\phi\otimes\phi',\varepsilon_{x,x'}\ra\neq 0\}$. One verifies the
compatibility with the set $\{(F,F')\in V_{\rho_G(x)}^{n+1}\times
V_{\rho_L(x')}^{n}|\Lambda_F\to\Lambda'_{F'}\}$ under the map $\rho$
as in the proof of Theorem \ref{prop-sc3}.
\end{proof}

\section{Springer correspondence for duals of symplectic and odd orthogonal Lie
algebras}\label{sec-d-c}
\subsection{}\label{ss-sym-d} We assume $G=Sp(2n)$ in this subsection and \ref{prop-sc1-d}.

Let
$\xi=(\lambda_{2s+1})_{\chi(\lambda_{2s+1})}\cdots(\lambda_1)_{\chi(\lambda_1)}\in\Lg^*$
be nilpotent (see \ref{comp-sym-d}), where $\lambda_1=0$. We have
$\lambda_{2j}=\lambda_{2j+1}$. We attach to the orbit $\rc$ of $\xi$
the sequence $c_1,\ldots,c_{2s+1}$ defined as follows:

(1) $c_{2j}=\lambda_{2j}-\chi(\lambda_{2j})+n+1+(j-1)(n+2)$,

(2) $c_{2j-1}=\chi(\lambda_{2j-1})+(j-1)(n+2)$.

Taking $a_i=c_{2i-1}$, $i=1,\ldots,s+1$, $b_i=c_{2i}$,
$i=1,\ldots,s$, we get a well-defined $(A,B)\in X_{n,1}^{1,n+1}$. We
denote it $\rho_G(\xi)$, $\rho(\xi)$ or $\rho(\rc)$.

\begin{lem}
$\mathrm{(i)}$ $\rc\mapsto\rho(\rc)$ defines a bijection from the
set of all nilpotent $Sp(2n)$-orbits in $\mathfrak{sp}(2n)^*$ to
$D_{n}^{1,n+1}$.

$\mathrm{(ii)}$ $A_G(\xi)^\wedge$ is isomorphic to
$V_{\rho(\xi)}^{1,n+1}$.
\end{lem}
\begin{proof}
(i) It is easily checked from the definition  that $\rho(\rc)\in
D_{n}^{1,n+1}$ and the map $\rc\mapsto\rho(\rc)$ is injective. Note
that $D_{n}^{1,n+1}$ is in bijection with the set $\Delta$
consisting of all pairs of partitions $(\mu,\nu)$ such that
$\sum\mu_i+\sum\nu_i=n,\nu_{i}\leq\mu_i+1$. Since the number of
nilpotent orbits is equal to $|\Delta|$ by \cite{X2}, the
bijectivity of the map follows. In fact, given $(A,B)\in
D_{n}^{1,n+1}$, the corresponding nilpotent orbit can be obtained as
follows. Let $c_1\leq c_2\leq\cdots\leq c_{2s+1}$ be the sequence
$a_1\leq b_1\leq\cdots\leq a_{s+1}$. Then
$\lambda_{2j}=\lambda_{2j+1}=c_{2j}+c_{2j+1}-(2j-1)(n+2)-(n+1)$ and
$\chi(\lambda_{2j})=\chi(\lambda_{2j+1})=c_{2j+1}-j(n+2)$,
$j=1,\ldots,s$, $\lambda_1=0$.

(ii) The component group $A_G(\xi)$ is described in
\ref{comp-sym-d}. Let $(A,B)=\rho(\xi)$ and $c_1,\ldots,c_{2s+1}$ be
as above. Let $S=(A\cup B)\backslash (A\cap B)$. Then
$\chi(\lambda_i)\neq(\lambda_i-1)/2$ if and only if $c_i\in S$. The
relation (r2) of \ref{comp-sym-d} says that if $c_i,c_j$ belong to
the same interval of $(A,B)$, then $\epsilon_i,\epsilon_j$ have the
same images in $A_G(\xi)$. Thus we get an element $\sigma_I$ of
$A_G(\xi)$ for each interval $I$ of $(A,B)$ and $\sigma_I^2=1$.
Moreover (r3) means that $\sigma_I=1$ if $I$ is the initial
interval.

The isomorphism $V_{\rho(\xi)}^{1,n+1}\rightarrow A_G(\xi)^\wedge$
is given as follows. We associate to $F$ the character of $A_G(\xi)$
which takes value $-1$ on $\sigma_I$ if and only if $I\in F$.
\end{proof}

\subsection{}\label{prop-sc1-d}Let $(\xi,\phi)\in\mathfrak{A}_{\Lg^*}$. We have defined
$\rho(\xi)$. Let $\rho$ denote also the map
$A_G(\xi)^\wedge\rightarrow V_{\rho(\xi)}^{1,n+1}$.

\begin{thm}
The Springer correspondence $\gamma':\mathfrak{A}_{\Lg^*}\rightarrow
\mathbf{W}_n^\wedge\cong X_n^{1,n+1}$ is given by
$$(\xi,\phi)\mapsto\rho(\xi)_{\rho(\phi)}.$$
\end{thm}

\begin{proof} By similar argument as in the proof of Theorem \ref{thm-sc-sp},
 it is enough to prove the map $\gamma'$ is compatible
with the restriction formula $(\mathbf{R}')$. Note that
$X_{n,d}^{1,n+1}=\emptyset$ if $d\neq 1$ is odd. Thus
$X_{n}^{1,n+1}=X_{n,1}^{1,n+1}$. Let $(A,B)\in X_{n}^{1,n+1}$
correspond to $\chi\in \mathbf{W}_n^\wedge$. The pairs $(A',B')\in
X_{n-1}^{1,n}$ which correspond to the components of the restriction
of $\chi$ to $\mathbf{W}_{n-1}$ are those which can be deduced from
$(A,B)$ by decreasing one of the entries $a_{i}$ by $i$ (or $b_{i}$
by $i+1$) and decreasing all other entries $a_j$ by $j-1$, $b_j$ by
$j$. We can decrease $a_{i}$ by $i$ (resp. $b_{i}$ by $i+1$) if and
only if $i\geq 2$, $a_i-a_{i-1}\geq n+3$ or $i=1$, $a_i\geq 1$
(resp. $i\geq2$, $b_i-b_{i-1}\geq n+3$ or $i=1$, $b_i\geq n+2$). We
write $(A,B)\to(A',B')$ if they are related in this way. Suppose
that $(A,B)\to(A',B')$. One can easily check that if $(A,B)$ and
$(A',B')$ are similar to $\Lambda\in D_n^{1,n+1}$ and $\Lambda'\in
D_{n-1}^{1,n}$ respectively, then $\Lambda\to\Lambda'$.

Let $\xi\in\Lg^*$ and $\xi'\in\Ll^*$ be nilpotent. Then
$S_{\xi,\xi'}\neq\emptyset$ if and only if $\xi,\xi'$ are as in
Proposition \ref{prop-dsp} if and only if
$\Lambda=\rho_G(\xi)\to\Lambda'=\rho_L(\xi')$. The verification of
compatibility with restriction formula is entirely similar to the
proof of Theorem \ref{prop-sc3}.
\end{proof}

\subsection{}\label{lem-com-dor}We assume $G=O(2n+1)$ in this subsection and \ref{prop-sc1}.

Let
$\xi=(m;(\lambda_{2s})_{\chi(\lambda_{2s})}\cdots(\lambda_1)_{\chi(\lambda_1)})\in\Lg^*$
be nilpotent (see \ref{ssec-dor}). We have
$\lambda_{2j-1}=\lambda_{2j}$. We attach to the orbit $\rc$ of $\xi$
the sequence $c_1,\ldots,c_{2s+1}$ defined as follows:

(1) $c_{2j}=\chi(\lambda_{2j})+(j-1)(n+1)$, $j=1,\ldots,s$

(2) $c_{2j-1}=\lambda_{2j-1}-\chi(\lambda_{2j-1})+(j-1)(n+1)$,
$j=1,\ldots,s$,

(3) $c_{2s+1}=m+s(n+1)$.

Taking $a_i=c_{2i-1}$, $i=1,\ldots,s+1$, $b_i=c_{2i}$,
$i=1,\ldots,s$, we get a well defined element $\{A,B\}\in
Y_{n,1}^{n+1}$. We denote it $\rho_G(\xi)$, $\rho(\xi)$ or
$\rho(\rc)$.

\begin{lem}
$\mathrm{(i)}$ $\rc\mapsto\rho(\rc)$ defines a bijection from the
set of all nilpotent $O(2n+1)$-orbits in $\mathfrak{o}(2n+1)^*$ to
$D_{n,\text{odd}}^{n+1}$.

$\mathrm{(ii)}$ $A_G(\xi)^\wedge$ is isomorphic to
$V_{\rho(\xi)}^{n+1}$.
\end{lem}
\begin{proof}
(i) It is easily checked from the definition  that $\rho(\rc)\in
D_{n,\text{odd}}^{n+1}$ and the map $\rc\mapsto\rho(\rc)$ is
injective. Note that $D_{n,\text{odd}}^{n+1}$ is in bijection with
the set $\Delta$ consisting of all pairs of partitions $(\mu,\nu)$
such that $\sum\mu_i+\sum\nu_i=n,\mu_{i+1}\leq\nu_i$. Since the
number of nilpotent orbits is equal to $|\Delta|$ by \cite{X2}, the
bijectivity of the map follows. In fact, given $\{A,B\}\in
D_{n,\text{odd}}^{n+1}$ with inverse image $(A,B)\in
D_{n,1}^{n+1,0}$, the corresponding nilpotent orbit can be obtained
as follows. Let $c_1\leq c_2\leq\cdots\leq c_{2s+1}$ be the sequence
$a_1\leq b_1\leq\cdots\leq a_{s+1}$. Then
$\lambda_{2j}=\lambda_{2j-1}=c_{2j}+c_{2j-1}-(2j-2)(n+1)$,
$\chi(\lambda_{2j})=\chi(\lambda_{2j-1})=c_{2j}-(j-1)(n+1)$,
$j=1,\ldots,s$ and $m=c_{2s+1}-s(n+1)$. The corresponding orbit is
$(m;(\lambda_{2s})_{\chi(\lambda_{2s})}\cdots(\lambda_1)_{\chi(\lambda_1)})$.

(ii) The component group $A_G(\xi)$ is described in \ref{ssec-dor}.
Let $\{A,B\}=\rho(\xi)$, $(A,B)$ and $c_1,\ldots,c_{2s+1}$ be as
above. Let $S=(A\cup B)\backslash (A\cap B)$. Note that
$c_{2s}<c_{2s+1}$, thus $c_{2s+1}\in S$. For $i=1,\ldots,2s$,
$\chi(\lambda_i)\neq\lambda_i/2$ if and only if $c_i\in S$. The
relation (r2) of \ref{ssec-dor} says that for $1\leq i<j\leq2s$, if
$c_i,c_j$ belong to the same interval of $\{A,B\}$, then
$\epsilon_i,\epsilon_j$ have the same images in $A_G(\xi)$. Let
$I_0$ be the interval containing $c_{2s+1}$. The relation (r3) says
that $\epsilon_i=1$ if $c_i\in I_0$. Thus we get an element
$\sigma_I$ of $A_G(\xi)$ for each interval $I\neq I_0$ of $\{A,B\}$
and $\sigma_I^2=1$.

The isomorphism $V_{\rho(\xi)}^{n+1}\rightarrow A_G(\xi)^\wedge$ is
given as follows. Let $F\in
V_{\rho(\xi)}^{n+1}=\mathcal{A}(E)/\{\emptyset,E\}$ and $\tilde{F}$
the inverse image of $F$ in $\mathcal{A}(E)$ that does not contain
$I_0$. We associate to $F$ the character of $A_G(\xi)$ which takes
value $-1$ on $\sigma_I$ if and only if $I\in \tilde{F}$.
\end{proof}

\subsection{}\label{prop-sc1}Let $(\xi,\phi)\in\mathfrak{A}_{\Lg^*}$. We have defined
$\rho(\xi)$. Let $\rho$ denote also the map
$A_G(\xi)^\wedge\rightarrow V_{\rho(\xi)}^{n+1}$.

\begin{thm}
The Springer correspondence $\gamma':\mathfrak{A}_{\Lg^*}\rightarrow
\mathbf{W}_n^\wedge\cong Y_{n,\text{odd}}^{n+1}$ is given by
$$(\xi,\phi)\mapsto\rho(\xi)_{\rho(\phi)}.$$
\end{thm}

\begin{proof} Again it is enough to prove the map $\gamma'$ is compatible
with the restriction formula $(\mathbf{R}')$. Note that
$Y_{n,d}^{n+1}=\emptyset$ if $d\neq 1$ is odd. Thus
$Y_{n,\text{odd}}^{n+1}=Y_{n,1}^{n+1}$. Let $\{A,B\}\in
Y_{n,\text{odd}}^{n+1}$ with inverse image $(A,B)\in
X_{n,1}^{n+1,0}$ correspond to $\chi\in \mathbf{W}_n^\wedge$. The
pairs $\{A',B'\}\in Y_{n-1,\text{odd}}^{n}$ with inverse images
$(A',B')\in X_{n-1,1}^{n,0}$ which correspond to the components of
the restriction of $\chi$ to $\mathbf{W}_{n-1}$ are those which can
be deduced from $(A,B)$ by decreasing one of the entries $a_{i}$ by
$i$ (or $b_{i}$ by $i$) and decreasing all other entries $a_j$ by
$j-1$, $b_j$ by $j-1$. We can decrease $a_{i}$ by $i$ (resp. $b_{i}$
by $i$) if and only if $i\geq 2$, $a_i-a_{i-1}\geq n+2$ or $i=1$,
$a_i\geq 1$ (resp. $i\geq2$, $b_i-b_{i-1}\geq n+2$ or $i=1$,
$b_i\geq 1$). We write $\{A,B\}\to\{A',B'\}$ if they are related in
this way. Suppose that $\{A,B\}\to\{A',B'\}$. One can easily check
that if $\{A,B\}$ and $\{A',B'\}$ are similar to $\Lambda\in
D_{n,\text{odd}}^{n+1}$ and $\Lambda'\in D_{n-1,\text{odd}}^{n}$
respectively, then $\Lambda\to\Lambda'$.

Let $\xi\in\Lg^*$ and $\xi'\in\Ll^*$ be nilpotent. Then
$S_{\xi,\xi'}\neq\emptyset$ if and only if $\xi,\xi'$ are as in
Proposition \ref{prop-1-d} if and only if
$\Lambda=\rho_G(\xi)\to\Lambda'=\rho_L(\xi')$. To verify the map is
compatible with the restriction formula, it is enough to show that
the set \begin{equation}\label{eqn-10}\{(F,F')\in
V_{\rho_G(\xi)}^{n+1}\times
V_{\rho_L(\xi')}^{n}|\Lambda_F\to\Lambda'_{F'}\}\end{equation} is
the image of the set \begin{equation}\label{eqn-11}
\{(\phi,\phi')\in A_G(\xi)^\wedge\times A_L(\xi')^\wedge|\la
\phi\otimes\phi',\varepsilon_{\xi,\xi'}\ra\neq0\}\end{equation}
under the map $\rho$.

Let $c_1\leq\cdots\leq c_{2s+1}$ and $c_1'\leq\cdots\leq c_{2s+1}'$
correspond to the pre-image $(A,B)$ and $(A',B')$ of $\Lambda$ and
$\Lambda'$ in $D_{n,1}^{n+1,0}$ and $D_{n-1,1}^{n,0}$ respectively.
Then $A_G(\xi)$ is generated by $\{\epsilon_i|c_i\neq c_j,\forall
j\neq i\}$, $A_L(\xi')$ is generated by $\{\epsilon_i'|c_i'\neq
c_j',\forall j\neq i\}$ and $A_P$ is generated by
$\{\epsilon_i|c_i\neq c_j,c_i'\neq c_j',\forall j\neq i\}$. There
are various cases to consider. We describe one of the cases in the
following and the other cases are similar.

Assume $k\geq 1$, $c_{2k}>c_{2k-1}+1$, $c_{2k+1}=c_{2k}+n$ and
$c_{2k}'=c_{2k}-k$, $c_{2i}'=c_{2i}-(i-1)$,$i\neq k$,
$c_{2i+1}'=c_{2i+1}-i$, $i=1,\ldots,s$. Let $I$ (resp. $I'$) be the
interval of $\Lambda$ (resp. $\Lambda'$) containing $c_{2k+1}$
(resp. $c_{2k+1}'$) and $J'$ the interval of $\Lambda'$ containing
$c_{2k}'$. Note that
$c_{2j}-c_{2j-1}=2\chi(\lambda_{2j})-\lambda_{2j}<n$ and
$c_{2j+1}-c_{2j}=n+1+\lambda_{2j+1}-\chi(\lambda_{2j+1})-\chi(\lambda_{2j})\geq
2$, except if $m=0$, $\chi(\lambda_{2s})=\lambda_{2s}=n$. In the
latter case, $\xi'$ correspond to $m'=0$,
$\chi'(\lambda_{2s}')=\lambda_{2s}'=n-1$ and $A_G(\xi)=A_L(\xi')=1$.
Hence all other intervals of $\Lambda$ and $\Lambda'$ can be
identified naturally. Let $I_0$ (resp. $I_0'$) be the interval of
$\Lambda$ (resp. $\Lambda'$) containing $c_{2s+1}$ (resp.
$c_{2s+1}'$) There are two possibilities:

(i) $I\neq I_0$. Let $\tilde{F}$ (resp. $\tilde{F}'$) be the
pre-image of $F$ (resp. $F'$) in $\mathcal{A}(E)$ (resp.
$\mathcal{A}(E')$) that does not contain $I_0$ (resp. $I_0'$). Then
$\Lambda_F\to\Lambda_{F'}$ if and only if
\begin{itemize}
\item[(a)]$\tilde{F}\backslash\{I\}=\tilde{F}'\backslash\{I',J'\}$;
\item[(b)]$\tilde{F}\cap\{I\}=\tilde{F}'\cap\{I',J'\}=\emptyset$ or $\{I\}\subset \tilde{F},\{I',J'\}\subset
\tilde{F}'$.
\end{itemize}
On the other hand, $A_G(\xi)$ (resp. $A_L(\xi')$) is an $\tF_2$
vector space with one basis element $\sigma_K$ (resp. $\sigma_{K}'$)
for each interval $K\neq I_0$ (resp. $K\neq I_0'$) of $\Lambda$
(resp. $\Lambda'$) and $S_{\xi,\xi'}$ is the quotient of
$A_G(\xi)\times A_L(\xi')$ by the subgroup $H_{\xi,\xi'}$ generated
by elements of the form
$\sigma_I\sigma'_{I'},\sigma_I\sigma'_{J'},\sigma_K\sigma'_K$ with
$K\neq I_0$ an interval of both $\Lambda$  and $\Lambda'$. Now the
compatibility between (\ref{eqn-10}) and (\ref{eqn-11}) is clear.

(ii) $I=I_0$. Then $\Lambda_F\to\Lambda_{F'}$ if and only if
$\tilde{F}=\tilde{F}'$. On the other hand, $A_G(\xi)$, $A_L(\xi')$
and $S_{\xi,\xi'}$ are obtained by setting $\sigma_I=\sigma'_{I'}=1$
in (i). Again the compatibility between (\ref{eqn-10}) and
(\ref{eqn-11}) is clear.
\end{proof}

\section{Complement}
\subsection{}In \cite{Lu5}, Lusztig gives an apriori description of the Weyl group representations that
parametrize the pairs $(\rc,1)$, where $\rc$ is a unipotent class in
$G$ or a nilpotent orbit in $\Lg$. We list the results of \cite{Lu5}
here.

Let $R$ be a root system of type $B_n,C_n$ or $D_n$ with simple
roots $\Pi$ and $\mathbf{W}$ the weyl group. There exists a unique
$\alpha_0\in R\backslash \Pi$ such that $\alpha-\alpha_i\notin
R,\forall\ \alpha_i\in \Pi$. Let $J\subset\Pi\cup\{\alpha_0\}$ be
such that $J=|\Pi|$. Let ${\mathbf{W}}_J$ be the subgroup of
${\mathbf{W}}$ generated by $s_{\alpha},\alpha\in J$.

(i) Denote $\mathcal{S}_{\mathbf{W}}$ the set of special
representations of $\mathbf{W}$. The set of unipotent classes when
$\text{char}(\tk)\neq 2$  is in bijection
 with the set (see
\cite{Lu6,Lu1}) \beq
\mathcal{S}^1_{\mathbf{W}}=\{j_{{\mathbf{W}}_J^*}^{\mathbf{W}}E,\
E\in\mathcal{S}_{ {\mathbf{W}}_J^*}\},\eeq where ${\mathbf{W}}_J^*$
is defined as ${\mathbf{W}}_J$ by taking $\alpha_0$ such that
$\check{\alpha}_0-\check{\alpha}_i\notin \check{R},\forall\
\alpha_i\in \Pi$ ($\check{R}$ is the coroot lattice and
$\check{\alpha}_0,\check{\alpha}_i$ are coroots).

(ii) The set of unipotent classes when $\text{char}(\tk)=2$ is in
bijection with the set (see \cite{Lu2}) \beq
\mathcal{S}^2_{\mathbf{W}}=\{j_{{\mathbf{W}}_J}^{\mathbf{W}}E,\
E\in\mathcal{S}^1_{ {\mathbf{W}}_J}\}.\eeq

(iii) The set of nilpotent classes when $\text{char}(\tk)=2$ is in
bijection with the set $\mathcal{T}_{\mathbf{W}}^2$ defined by
induction on $|{\mathbf{W}}|$ as follows (see \cite{Lu4}). If
${\mathbf{W}}=\{1\}$,
$\mathcal{T}_{\mathbf{W}}^2={\mathbf{W}}^\wedge$. If
${\mathbf{W}}\neq\{1\}$, then $\mathcal{T}_{\mathbf{W}}^2$ is the
set of all $E\in {\mathbf{W}}^\wedge$ such that either
$E\in\mathcal{S}_{\mathbf{W}}^1$ or
$E=j_{{\mathbf{W}}_J}^{\mathbf{W}}E_1$ for some ${\mathbf{W}}_J\neq
{\mathbf{W}}$ and some $E_1\in\mathcal{T}^2_{ {\mathbf{W}}_J}$.

\subsection{}One can show that the set of nilpotent orbits in $\Lg^*$
when $\text{char}(\tk)=2$ is in bijection with the set
$\mathcal{T}_{\mathbf{W}}^{2*}$ defined by induction on
$|\mathbf{W}|$ as follows. If ${\mathbf{W}}=\{1\}$,
$\mathcal{T}_{\mathbf{W}}^{2*}={\mathbf{W}}^\wedge$. If
${\mathbf{W}}\neq\{1\}$, then $\mathcal{T}_{\mathbf{W}}^{2*}$ is the
set of all $E\in {\mathbf{W}}^\wedge$ such that either
$E\in\mathcal{S}_{\mathbf{W}}^1$ or
$E=j_{{\mathbf{W}}_J^*}^{\mathbf{W}}E_1$ for some
${\mathbf{W}}_J^*\neq {\mathbf{W}}$ and some
$E_1\in\mathcal{T}^{2*}_{ {\mathbf{W}}_J^*}$.

\vskip 10pt {\noindent\bf\large Acknowledgement} \vskip 5pt I would
like to thank Professor George Lusztig for guidance, encouragement
and many helpful discussions, and thank the referees for comments.

\end{document}